\documentclass{article}

\PassOptionsToPackage{square,numbers,compress}{natbib}

     \usepackage[preprint]{neurips_2019}

\usepackage[utf8]{inputenc} 
\usepackage[T1]{fontenc}    
\usepackage{hyperref}       
\usepackage{url}            
\usepackage{booktabs}       
\usepackage{amsfonts, amsthm, amssymb}       
\usepackage{nicefrac}       
\usepackage{microtype}      

\bibpunct{(}{)}{;}{a}{,}{,}
\usepackage{hyperref}
\hypersetup{
  colorlinks = true,
  urlcolor = blue,
  linkcolor = blue,
  citecolor = blue,
}

\usepackage{bm}             

\usepackage[lined,noend]{algorithm2e}

\usepackage{mathtools}
\usepackage{enumerate}
\usepackage{dsfont}
\usepackage{graphicx}
\usepackage{enumitem}
\usepackage{wrapfig}
\usepackage{cleveref, nameref}
\usepackage{autonum}
\usepackage{thmtools}
\usepackage{multirow}

\usepackage{pgf,tikz,pgfplots}
\pgfplotsset{compat=1.14}
\usepackage{mathrsfs}
\usetikzlibrary{arrows}

\newtheorem{theorem}{Theorem}
\newtheorem{assumption}{Assumption}
\newtheorem{remark}{Remark}
\newtheorem{lemma}{Lemma}
\crefname{assumption}{assumption}{assumptions}
\numberwithin{assumption}{section}
\newtheorem{lem}{Lemma}
\numberwithin{lem}{section}
\newenvironment{myproof}[1][Proof]%
  {%
   \par\noindent{\bfseries\upshape #1\ }%
  }%
  {\qed}

\newcommand{\R}{\mathbb R}

\newcommand{\PP}{\mathbb P}

\newcommand{\E}{\mathbb E}

\newcommand*\diff{\mathop{}\!\mathrm{d}}
\newcommand{\DS}{\displaystyle}
\newcommand{\bs}{\boldsymbol}
\def\simiid{\overset{\text{iid}}{\sim}}
\newcommand{\prt}[1]{\left( \, #1  \, \right)}
\newcommand{\crc}[1]{\left[ \, #1  \, \right]}

\def\hat{\widehat}

\title{A nonasymptotic law of iterated logarithm for general $M$-estimators}

\author{%
  Victor-Emmanuel~Brunel,
  Arnak~Dalalyan,
  Nicolas~Schreuder \\
  CREST, ENSAE\\
  Palaiseau, FRANCE\\
  \texttt{nicolas.schreuder@ensae.fr} \\
}

\begin{document}

\maketitle

\begin{abstract}
$M$-estimators are ubiquitous in machine learning and statistical
learning theory. They are used both for defining prediction strategies 
and for evaluating their precision. In this paper, we propose the first 
non-asymptotic ``any-time'' deviation bounds for general $M$-estimators, 
where ``any-time'' means that the bound holds with a prescribed probability 
for every sample size. These bounds are non-asymptotic versions of the 
law of iterated logarithm. They are established under general assumptions 
such  as Lipschitz continuity of the loss function and (local) curvature 
of the population risk. These conditions are satisfied for most examples 
used in machine learning, including those ensuring robustness to outliers 
and to heavy tailed distributions. As an example of application, we consider
the problem of best arm identification in a parametric stochastic multi-arm 
bandit setting. We show that the established bound can be converted into
a new algorithm, with provably optimal theoretical guarantees. Numerical 
experiments illustrating the validity of the algorithm are reported.
\end{abstract}

\section{Introduction}\label{sec:1}
Perhaps the most fundamental theorems in statistics are the law of large numbers (LLN) and the central limit theorem (CLT). Morally, they state that a sample average converges almost surely or in probability to the population average, and if one zooms in by multiplying by a square root factor, a much weaker form of stochastic convergence still holds, namely, convergence in distribution towards a Gaussian law. A fine intermediate result shows what happens in between the two scales: the law of iterated logarithm (LIL). By zooming in slightly less than in the CLT, \textit{i.e.}, by rescaling the sample average with a slightly smaller factor than in the CLT, it is possible to gain a guarantee for infinitely many sample sizes, almost surely. In practice, however, the LIL has limited applicability, since it does not specify for which sample sizes the guarantee holds. The goals of the present work are 
(a) to lift this limitation, by proving a LIL valid for every sample size, and (b) to 
extend the LIL (known to be true for sample averages) to general $M$-estimators.   

The precise statement of the LIL, discovered by \citet{Khintchine1924,Kolmogoroff1929} 
almost a century ago, is as follows: For a sequence of 
iid random variables  $\{Y_i\}_{i\in\mathbb{N}}$
with mean $\theta$ and variance $\sigma^2<\infty$, the sample averages 
$\Bar Y_n = (Y_1+\ldots+Y_n)/n$ 
satisfy the relations
\begin{align}
\liminf_{n\to\infty} \frac{\sqrt{n}\,(\Bar Y_n-\theta)}{
\sigma\sqrt{2\ln\ln n}} =-1 
\quad\text{and}\quad
\limsup_{n\to\infty} \frac{\sqrt{n}\,(\Bar Y_n-\theta)}{
\sigma\sqrt{2\ln\ln n}} =1,\quad 
\text{almost surely}. 
\end{align}
This provides a guarantee on the deviations of the sample 
average as an estimator of
the mean $\theta$, since it yields that with probability one, 
there is a $n_0\in\mathbb{N}$ such that $|\Bar{Y}_n-\theta|
\le \sigma (\nicefrac{2\ln\ln n}{n})^{1/2}$ for every $n\ge n_0$. 
As compared to the deviation guarantees provided 
by the central limit theorem, the one of
the last sentence has the advantage of being valid for 
any sample size large enough. This advantage is gained 
at the expense of a factor $(\ln\ln n)^{1/2}$. Akin for 
the classic version of the CLT, the applicability of the 
LIL is limited by the fact that it is hard to get
any workable expression of $n_0$. 
 
In the case of the CLT and its use in statistical learning, 
the drawback related to $n_0$ was lifted by exploiting 
concentration inequalities, such as the Hoeffding or the 
Bernstein inequalities, that can be seen as non-asymptotic 
versions of the CLT. For bounded random variables, the 
aforementioned concentration inequalities imply that for 
a prescribed tolerance level $\delta\in(0,1)$, for every 
$n\in\mathbb{N}$, the event\footnote{Here $C$ is a 
universal constant.}  $\mathcal{A}_n=\{|\Bar{Y}_n-\theta|
\le C(\nicefrac{\ln(1/\delta)}{n})^{1/2}\}$ holds with 
probability at least $1-\delta$. Such a deviation bound 
is satisfactory in a batch setting, when all the 
data are available in advance. In contrast, when data 
points are observed sequentially, as in on-line learning, 
or when the number of acquired data points depends on the
actual values of the data points, the event of interest 
is $\Bar{\mathcal A}_{N} = \mathcal{A}_1\cap \ldots\cap
\mathcal{A}_N$ or even a version of it in which $N$ can be
replaced by $\infty$. One can use the union bound to ensure
that $\Bar{\mathcal A}_{N}$ has a probability at least 
$1-N\delta$ but this is too crude. Furthermore, replacing in 
$\mathcal A_n$ the confidence $\delta$ by $\delta/n^2$, 
we get coverage $1-\frac{\pi^2}{6}\delta$, valid for 
any sample size $n$ for an interval of length 
$O((\nicefrac{\ln n}{n})^{1/2})$. This result, obtained
by a straightforward application of the union bound, is 
sub-optimal. A remedy to such a sub-optimality---in the 
form of a nonasymptotic version of the LIL---was proposed by
\cite{jamieson2014lil} and further used by 
\cite{Kaufman16,Kaufmann_Koolen,howard2018uniform}. 
In addition, its relevance for 
online learning was demonstrated by deriving guarantees 
for the best arm selection in a multi-armed bandit setting. 
Note that these recent results apply exclusively to the 
sample mean; there is no equivalent of these bounds for 
other types of estimators. 

In this work, we establish a non-asymptotic LIL in a general 
setting encompassing many estimators, far beyond the sample 
average. More precisely, we focus on  the class of (penalized) 
$M$-estimators comprising the sample average but also the sample 
median, the  quantiles, the least-squares estimator, etc. Of 
particular interest to us are estimators that are robust to 
outliers and/or to heavy tailed distributions. This is the case 
of the median, the quantiles, the Huber estimator, etc. 
\citep{huber1964robust, Huber_book}. It is well known that 
under mild assumptions, $M$-estimators are both consistent 
and asymptotically normal, \textit{i.e.}, a suitably adapted 
version of the LLN and the CLT applies to them \citep{vanderVaart,
Portnoy,Collins}. Moreover, some versions of the LIL were 
also shown for $M$-estimators \citep{arcones1994some,he1995law}, 
with little impact in statistics and machine learning, because 
of the same limitations as those explained above for the 
standard LIL. Our contributions complement these studies by 
providing a general non-asymptotic LIL for $M$-estimators.  

We apply the developed methodology to the problem of multi-armed 
bandits when the rewards are heavy tailed or contaminated by 
outliers. In such a context, \cite{Brunel} tackled the problem 
of best median arm identification; this corresponds to
replacing the average regret by the median regret. The relevance of this approach relies 
on the fact that even a small number of contaminated 
samples obtained from each arm may make the corresponding 
means arbitrarily large. The method proposed in \cite{Brunel} 
is a suitable adaptation of the well-known upper 
confidence band (UCB) algorithm. In that setup, would 
it be possible to improve the upper bounds on the sample 
complexity of their algorithm---similarly 
to \cite{jamieson2014lil}---by using some version of 
the uniform LIL for empirical medians or, more generally, 
for robust estimators? Our main results yield a positive 
answer to this question. 

The rest of the paper is organized as follows. The next 
section contains the statement of the LIL in a univariate 
setting and provides some examples satisfying the required 
conditions. A mutlivariate version of the LIL for penalized 
$M$-estimators is presented in \Cref{sec:3}. An application 
to on-line learning is carried out in \Cref{sec:4}, while
a summary of the main contributions and some future 
directions of research are outlined in \Cref{sec:5}. 
Detailed proofs are deferred to the supplementary material.

\section{Uniform law of iterated logarithm for \texorpdfstring{$M$}{M}-estimators}\label{sec:2}



In this section, we focus on the case of univariate $M$-estimators, 
which are a natural extension of the empirical mean, 
especially in robust setups (see \cite{huber1964robust,Maronna}
as well as the recent work by \cite{loh2017} and the references 
therein). We consider a sequence $Y,Y_1,Y_2,Y_3,\ldots$ of i.i.d.\ 
random variables in some arbitrary space $\mathcal Y$ with probability distribution $\PP_Y$ and we let 
$\phi:\mathcal Y\times \Theta\rightarrow\R$ be a given loss function, where $\Theta$ is an open interval in $\R$. We make the two following assumptions on the loss $\phi$.
\begin{assumption}
\label{as:finite_expectation}
For all $\theta\in\Theta$, the random variable $\phi(Y, \theta)$ has a finite expectation.
\end{assumption}
\begin{assumption}
\label{as:convex_phi}
The function $\phi(Y,\cdot)$ is convex 
$\PP_Y$-almost surely and $\phi(Y,\theta)
\rightarrow\infty$ as $\theta$ approaches the boundary of $\Theta$, $\PP_Y$-almost surely (we say that the $\phi(Y,\cdot)$ is convex and coercive).
\end{assumption}
We define the population risk $\Phi(\theta)=
\E\crc{\phi(Y,\theta)}$  and, for all integers $n\geq 1$, 
the empirical risk $\hat\Phi_n(\theta)=\frac{1}{n}\sum_{i=1}^n
\phi(Y_i,\theta)$.
We denote by $\theta^*$ a minimizer of $\Phi$ on $\Theta$, and by 
$\hat\theta_n$ a minimizer of $\hat\Phi_n$ on $\Theta$, for all 
$n\geq 1$. \Cref{as:convex_phi} requires from the loss $\phi$ to approximately have a U-shape in order to guarantee that the quantities $\theta^*$ and $\hat{\theta}_n$ are well defined. We need two more 
assumptions to state our result.
\begin{assumption}
\label{as:Phi_strongly_convex}
The minimizer $\theta^*$ of $\Phi$ is unique and there exist two positive constants $r$ and $\alpha$ such that for all $\theta\in\Theta$ with $|\theta-\theta^*|\leq r$, $\DS \Phi(\theta)\geq \Phi(\theta^*)+(\nicefrac{\alpha}{2})(\theta-\theta^*)^2$.
\end{assumption}

\begin{assumption}
\label{as:SG}
There exists a positive constant $\sigma^2$ such that the random variables $\phi(Y,\theta)-\phi(Y,\theta^*)$ are $\sigma^2(\theta-\theta^*)^2$-sub-Gaussian\footnote{See, e.g., \citep[Section 3.1]{koltchinskii} for a definition of centered sub-Gaussian random variables and their properties. A non-zero mean random variable is sub-Gaussian if its centered version is sub-Gaussian.} 
for all $\theta\in\Theta$.
\end{assumption}
\Cref{as:Phi_strongly_convex} requires from $\Phi$ to have a positive curvature in
a neighborhood of the oracle $\theta^*$. It is weaker than the local strong convexity
of $\Phi$. \Cref{as:SG} is a smoothness condition on $\phi(Y,\cdot)$. In particular, it is fulfilled if $\phi(Y,\cdot)$ is $\eta$-Lipschitz with a sub-Gaussian variable $\eta$.  
We stress that the function $\phi$ is not assumed differentiable and that $Y$ is not necessarily sub-Gaussian. We are now ready to state our first theorem on the uniform concentration of $M$-estimators.

\begin{theorem} \label{MainThm1d}
	Let \Cref{as:finite_expectation,as:convex_phi,as:Phi_strongly_convex,as:SG} hold. Then, for any $\delta\in (0, 1)$,
	\begin{equation}\label{main:1}
	\mathbb{P} \left(\forall n\ge n_0,  \quad \lvert\hat\theta_n-\theta^*\rvert\leq 
		t_{n,\delta}^{\rm LIL}:=\frac{3.4\sigma}{\alpha}
		\sqrt{\frac{\ln\ln 2n +	0.72\ln(\nicefrac{10.4}{\delta})}{n}}
		\right) \geq 1 - \delta,
	\end{equation}
where $n_0=n_0(\alpha,r,\delta)$ is the smallest integer $n\ge 1$ for which $t_{n,\delta}^{\rm LIL}
\le r$.
\end{theorem}

\begin{remark} \label{remark:relax}
	In the definition of $\Phi$ and $\hat\Phi_n$, one can replace $\phi(Y,\theta)$ with $\phi(Y,\theta)-\phi(Y,\theta_0)$ for any arbitrary $\theta_0\in\Theta$, without changing the values of $\theta^*$ and $\hat\theta_n$. Then, Assumption \ref{as:finite_expectation} becomes less restrictive for $Y$ in general, since it only requires $\phi(Y,\theta)-\phi(Y,\theta_0)$ to have a finite expectation. For instance, for median estimation, $\phi(Y,\theta)=|Y-\theta|$, yet the median should be defined even if $Y$ does not have an expectation. Taking $\theta_0=0$ yields $\phi(Y,\theta)-\phi(Y,\theta_0)=|Y-\theta|-|Y|$, which is bounded, hence, always has an expectation.
\end{remark}

We now give some natural examples for which all the assumptions presented above are satisfied.

\paragraph{Mean estimation}

Let $\mathcal Y=\Theta=\R$ and $\phi(x,\theta)=(x-\theta)^2 $. Assume that $Y$ is 
$s^2$-sub-Gaussian. Then, it is easy to see that \Cref{as:finite_expectation,as:convex_phi,as:Phi_strongly_convex,as:SG} are all satisfied with $r=+\infty$, $\alpha=2$ and $\sigma=2s$. The standard deviation is doubled because of \Cref{as:SG}, it is the cost for the generality of our result. However it is not a problem since we want to focus on other M-estimators, the mean of sub-Gaussian variables being already well studied (see, e.g., \cite{howard2018uniform}).

\paragraph{Median and quantile estimation}

Let $\mathcal Y=\Theta=\R$ and $\phi(x,\theta)=|x-\theta|-|x|$. Assume that $Y$ has a unique median $\theta^*$ and that its cumulative distribution function $F$ satisfies $|F(\theta)-\nicefrac{1}{2}|\geq (\nicefrac{\alpha}{2})|\theta-\theta^*|$, for all $\theta\in [\theta^*-r,\theta^*+r]$, where $r>0$ is a fixed number. Then, $\theta^*$ is the unique minimizer of $\Phi$ and for all $\theta\in [\theta^*-r,\theta^*+r]$,
\begin{align}
	\Phi(\theta)-\Phi(\theta^*) & = 2\int_{(\theta^*,\theta]}x\diff F(x)-(\theta-\theta^*)+2(\theta F(\theta)-\theta^* F(\theta^*)) \\
	& = 2\int_{(\theta^*,\theta]}F(x)\diff x-(\theta-\theta^*) \geq \frac{\alpha}{2}(\theta-\theta^*)^2,
\end{align}
yielding Assumption \ref{as:Phi_strongly_convex}. Moreover, since $\phi(Y,\theta)$ is bounded almost surely and $1$-Lipschitz, for all $\theta\in\R$, Assumptions \ref{as:finite_expectation} and \ref{as:SG} are automatically true (with $\sigma =1$).

The same arguments hold true if $\phi(x,\theta)=\tau_\alpha(x-\theta)-\tau_\alpha(x)$, where $\tau_{\alpha}(x)=\alpha x$ if $x\geq 0$, $\tau_{\alpha}(x)=(\alpha-1)x$ otherwise, for which $\theta^*$ is the $\alpha$-quantile of $Y$, for $\alpha\in (0,1)$.

\paragraph{Huber's $M$-estimators}

Let $\mathcal Y=\Theta=\R$ and let $c>0$. Denote by $g_c(x)=x^2$ if $|x|\leq c$, $g_c(x)=c(2|x|-c)$ if $|x|>c$ and let $\phi(x,\theta)=g_c(x-\theta)-g_c(x)$. This function $g_c$ being $2c$-Lipschitz, \Cref{as:SG} is satisfied with $\sigma=2c$. Assume that $Y$ has a positive density $f$ on $\R$. Then, it is easy to check that $\Phi$ is twice differentiable, with $\Phi''(\theta)=2\prt{F(\theta+c)-F(\theta-c)}>0$, for all $\theta\in\R$, where $F$ is the cumulative distribution function of $Y$. Hence, $\theta^*$ is well-defined and unique, and if there exists $m>0$ such that $f(x)\geq m$ for $x\in [\theta^*-2c,\theta^*+2c]$, then Assumption \ref{as:Phi_strongly_convex} is satisfied with $r=2c$ and $\alpha=4cm$. 

\paragraph{Comparison between union bound and LIL}

Let $Y_1, \dots, Y_n$ be i.i.d. random variables and let $\phi : \mathbb{R}\times \mathbb{R} \xrightarrow{} \mathbb{R}$ be a loss such that assumptions of \Cref{MainThm1d} are satisfied. Let $\hat{\theta}_n$ be the $M$-estimator associated with the samples $Y_1, \dots, Y_n$ and the loss $\phi$. \Cref{lem:m_estimator_to_SG} in \Cref{sec:6} gives the following tail bound :
$\forall n \geq 1$, $\mathbb{P}\big( \lvert \hat{\theta}_n-\theta^* \rvert > \frac{2\sigma}{\alpha}  \sqrt{\nicefrac{2\ln(2/\delta)}{n}} \big) \leq \delta$.
A naive union bound then gives
\begin{equation} \label{eq:UB_upper_bound}
\mathbb{P}\bigg(\lvert \hat{\theta}_n - \theta^* \rvert \leq 
t_{n,\delta}^{\rm UB} \coloneqq \frac{2\sigma}{\alpha} 
\sqrt{\frac{2\ln (2n^{1+\varepsilon}/\delta)}{n}}\ \text{ for all }
n\ge 1\bigg) \geq 1 - \sum_{n=1}^\infty \frac{\delta}{n^{1+\varepsilon}} 
\geq  1 -  \zeta(1+\varepsilon)\delta.
\end{equation}
\begin{figure}
    \centering
    \includegraphics[width=0.6\linewidth]{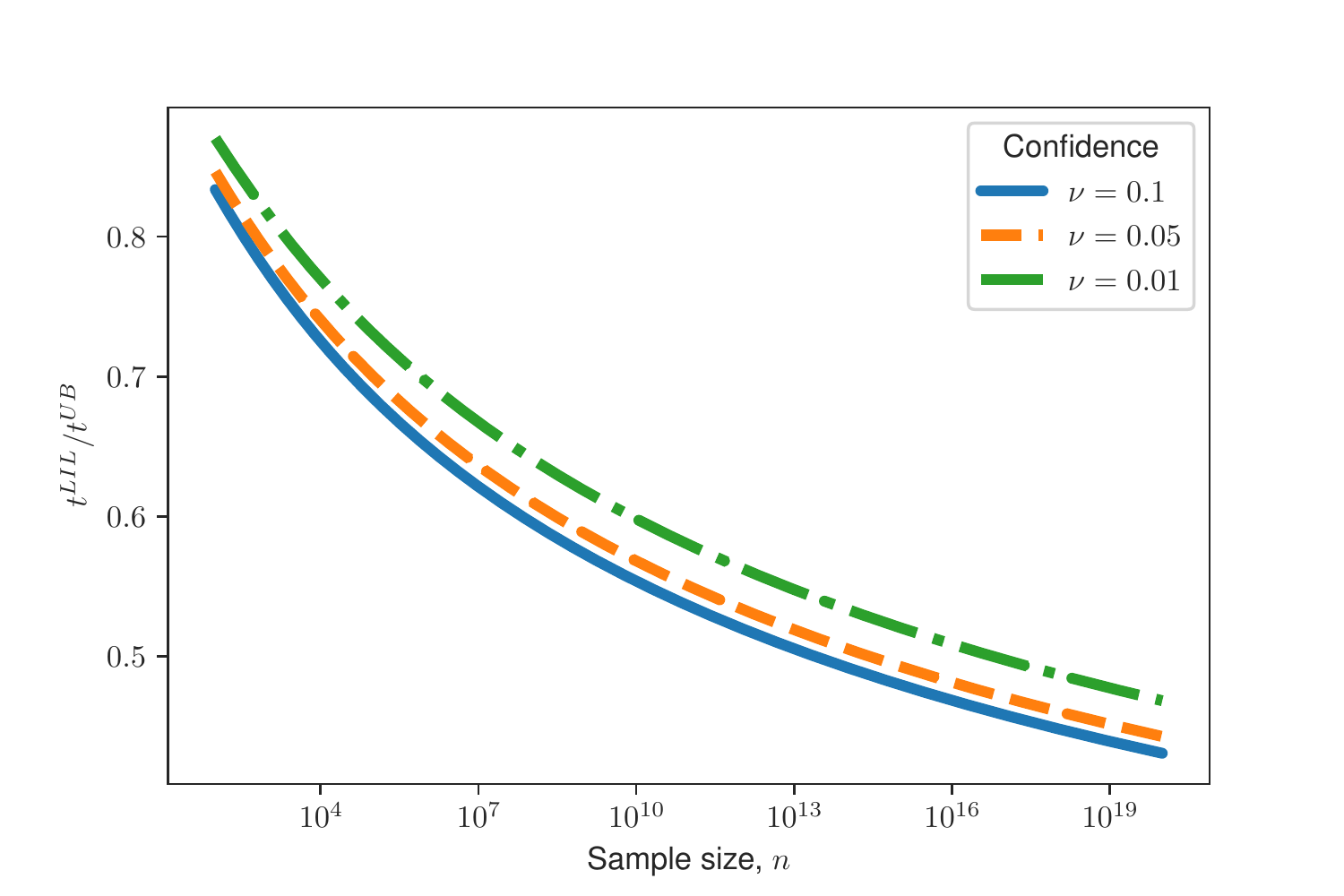}
    \caption{Ratio $t^{\rm LIL}_{n, \delta}/t^{UB}_{n, \delta'}$ for different sample sizes $n$ and confidence levels $\nu$.}
    \label{fig:vs_UB}
\end{figure}
\Cref{fig:vs_UB} shows the ratio of the LIL upper 
bound $t_{n,\delta}^{\rm LIL}$ provided by \Cref{MainThm1d} 
over the sub-Gaussian upper bound $t_{n,\delta'}^{\rm UB}$ 
for different levels of global confidence. The parameters 
$\delta$ and $\delta'$ are chosen to guarantee that the 
right hand sides in both \eqref{main:1}  and 
\eqref{eq:UB_upper_bound} are equal to the prescribed 
confidence level. For $t_{n,\delta'}^{\rm UB}$, we chose 
$\varepsilon = 0.1$, the results for other values of 
$\varepsilon$ being very similar. We observe that the LIL 
bound is always better than the one obtained by the union 
bound. In addition,  the gap between the bounds widens as 
the sample size grows.

\section{Uniform LIL for \texorpdfstring{$M$}{M}-estimators of a multidimensional parameter}

\def\bX{\bs{X}}
\def\btheta{\bs{\theta}}

We consider here a standard setting in supervised learning, in which the goal is to predict
a real valued label using a $d$-dimensional feature.
More precisely, we are given $n$ independent label-feature pairs $(\bX_1, Y_1), \dots, (\bX_n, Y_n)$,
with labels $Y_i \in \mathbb{R}$ and features $\bX_i \in \mathbb{R}^d$, drawn from a common 
probability distribution $P$. Let $\phi : \mathbb{R} \times \mathbb{R} \to \mathbb{R}$ be a given 
loss function and $\rho_n:\mathbb{R}^d\to\mathbb{R}$ a given penalty. For a sample $(\bX_1, Y_1), \dots, 
(\bX_n, Y_n)$, we define the penalized empirical and population risks
\begin{align}
    \hat\Phi_n(\btheta) &=
    \frac1n \sum_{i=1}^n \phi(Y_i,\btheta^\top \bX_i)
    + \rho_n(\btheta)
    \quad\text{and}\quad
    \Phi_n(\btheta) =
    \mathbb{E} \big[ \phi(Y_1,\btheta^\top \bX_1) \big]
    + \rho_n(\btheta).
\end{align}
Note that the penalty $\rho_n$ is allowed to depend on the sample size $n$. Since our 
results are non-asymptotic, this dependence will be reflected in the constants appearing 
in the law of iterated logarithm stated below. We also
define the penalized $M$-estimator $\bm{\hat{\theta}}_n$ and its
population counterpart $\btheta^*$ by
\begin{align}\label{ERM}
    \bm{\hat{\theta}}_n \in \arg \min_{\btheta\in\mathbb{R}^d} \hat{\Phi}_n(\btheta)\quad\text{and}\quad
    \btheta^* \in \arg \min_{\btheta\in\mathbb{R}^d} \Phi_n(\btheta).
\end{align}
Typical examples where such a formalism is applicable are the maximum
a posteriori approach and penalized empirical risk minimization. 
Our goal is to establish a tight non-asymptotic bound on the error
of $\bm{\hat{\theta}}_n$, that is, with high probability, valid for every $n\in\mathbb{N}$.
To this end, we consider a unit vector $\bs{a} \in \mathbb{R}^d$ and we are 
interested in bounding the deviations of the random variable 
${\bs a}^\top(\bm{\hat{\theta}}_n - \btheta^*)$. One can think of $\bs{a}$ as 
the feature vector of a new example, the label of which is unobserved. 
We aim at providing uniform non-asymptotic guarantees on the quality 
of the predicted label $\hat y = \bs{a}^\top\hat\btheta_n$. 

The main result of this section is valid under the assumptions listed below. 
We will present some common examples in which all these assumptions are satisfied. 

\begin{assumption} \emph{(Finite expectation)}\label{as:exp}
The random variables $\phi(Y_1,\btheta^\top \bX_1)$ has a finite expectation,
for every $\btheta$, with respect to the probability distribution $P$.
\end{assumption}

\begin{assumption} \label{as:convex_lipschitz_phi} \emph{(Convex and Lipschitz loss)}
The function $u \mapsto \phi(y, u)$ is $L$-Lipschitz and convex for any $y\in \mathbb{R}$.
\end{assumption}

\begin{assumption}
\label{as:convex_penalty}\emph{(Convex penalty)}
The penalty $\theta \mapsto \rho_n(\theta)$ is a convex function.
\end{assumption}

\begin{remark}
\Cref{as:convex_lipschitz_phi,as:convex_penalty} can be replaced by the assumption that the function $\hat{\Phi}_n$ is convex almost surely. 
\end{remark}

\begin{assumption} \label{as:Phi_strongly_convex_everywhere} \emph{(Curvature of the population risk)}
There exists a positive non-increasing sequence $(\alpha_n)$ such that, for any $n \in \mathbb{N}^*$, for any ${\bs w} \in \mathbb{R}^d$,
$\Phi_n(\btheta^* + \bs{w}) - \Phi_n(\btheta^*) \geq (\nicefrac{\alpha_n}{2}) \lVert {\bs w} \rVert ^2_2$.
\end{assumption}

\begin{assumption}
\label{as:boundedness}\emph{(Boundedness of features)}
There exists a  positive constant $B$  such that
$\lVert \bX_1 \rVert_2 \leq B$ almost surely.
\end{assumption}
We will use the notation $\kappa_n = L/\alpha_n$ and refer to this quantity as the 
condition number. Note that all the foregoing assumptions are common in statistical 
learning, see for instance \citep{Karthik,Rakhlin2012}. They are helpful not only 
for proving statistical guarantees but also for designing efficient computational 
methods for approximating $\hat\btheta_n$. 

For instance, if $\rho_n(\btheta) = \lambda_n\|\btheta\|_2^2$ is the ridge penalty 
\citep{Ridge} and $\phi$ is either the absolute deviation ($\phi_{abs}(y,y') = |y-y'|$, see
for instance \citep{WangWY14}), the hinge ($\phi_{abs}(y,y') = (1-yy')_{+}$ with $y\in[-1,1]$)  
or the logistic ($\phi_{log}(y,y') = \ln(1+e^{-yy'})$ with $y\in [-1,1]$) loss, 
the aforementioned assumptions are satisfied with $L = 1$ and $\alpha_{n} = \lambda_n$. 
One can also consider the usual squared loss $\phi(y,y') = (y-y')^2$ under 
the additional assumption that $Y$ is bounded by a known constant $B_y$. 
In this condition, if the minimization problems in \eqref{ERM} are constrained to 
the ball of radius $R$, Assumptions \ref{as:convex_lipschitz_phi} and 
\ref{as:Phi_strongly_convex_everywhere} are satisfied with $\alpha_n = 1$ and $L = 2B_y + BR$. It should be noted that \Cref{as:Phi_strongly_convex_everywhere} is satisfied, for instance, when
$\Phi_n$ is strongly convex. Importantly, as opposed to some other papers \citep{Hsu}, 
we need this assumption for the population risk only. 

\begin{theorem}\label{main:thm2}
Let \Cref{as:exp,as:convex_lipschitz_phi,as:convex_penalty,as:Phi_strongly_convex_everywhere,as:boundedness} be satisfied for every $n\in\mathbb{N}$. Assume, in addition, that 
the sequence $\nicefrac{\ln\ln n}{n\alpha_n^2}$ is decreasing. Then, for any vector 
$\bs{a}\in\mathbb{R}^d$ and any $\delta \in (0, 1)$, 
\begin{align}
     \mathbb{P}\left(\forall n \geq 1, \quad {\bs a}^\top (\bm{\hat{\theta}}_n - \btheta^*) \leq \frac{10 B\kappa_n}{\sqrt{3}}\|\bs{a}\|_2
    \sqrt{\frac{1.2 \ln\ln n + \ln(3/\delta) + 3}{n}}\right) \geq 1 - \delta.
\end{align}
\end{theorem}

Conditions under which \Cref{main:thm2} holds can be further relaxed. We have namely 
in mind the following three extensions. First, \Cref{as:boundedness} can be 
replaced by sub-Gaussianity of $\bs{X}$. Second, the curvature condition can 
be imposed on a neighborhood of $\btheta^*$ only, by letting $\Phi_n$ grow 
linearly outside the neighborhood. Third, the Lipschitz assumption on $\phi$ 
can be replaced by the following one: for a constant $\beta$ and a sub-Gaussian 
random variable $\eta$, the function 
$u\mapsto \phi(Y,u)-\beta u^2$ is $\eta$ Lipschitz. This last extension will 
allow us to cover the case of squared loss without restriction to a bounded 
domain. All these extensions are fairly easy to implement, but they significantly 
increase the complexity of the statement of the theorem. 
In this work, we opted for sacrificing the generality in order to get better readability of the result.

Another interesting avenue for future research is the extension of the presented
results to high-dimensional on-line setting, \textit{i.e.}, when the dimension might
be larger than the sample size, see \citep{negahban2012} for an in-depth discussion of 
the batch setting. \label{sec:3}

\section{Application to Bandits}


In this section, we apply the univariate uniform law of iterated 
logarithm that we proved in \Cref{sec:2} to a problem of multi-armed
bandits in the fixed confidence setting. The Best Arm Identification 
(BAI) problem in the fixed confidence setting usually consists in 
identifying, as fast as possible, which arm produces the highest 
expected outcome, see e.g. \citep{audibert2010best,GabillonGL12,Kaufman16}. 
A more probabilistic formulation of the problem is the following: 
we are able to collect data by sampling from $K$ unknown distributions
$P_1,\ldots,P_K$, the goal is to identify the distribution having the
largest expectation. Naturally, the same problem can be formulated for 
finding the distribution with the largest median, or the largest 
quantile of a given order. In particular, such a formulation of the
problem might be of interest in cases where the expectations of the 
outcomes of each arm may not be defined (rewards are heavy tailed) or
are not meaningful (rewards are subject to some arbitrary contamination). 
Such a problem has been recently considered by \cite{Brunel}. From a
statistical perspective, the problem under consideration is to find
the maximum point in a quantile regression problem \citep{Chernozhukov05}. 
The theoretical results of previous sections allow us to adapt the 
LIL'UCB algorithm of \cite{jamieson2014lil} to this general framework. 

\paragraph{Robust BAI}
We consider a robust version of BAI, which we call Robust BAI (RBAI). Let $(P_\theta)_{\theta\in\R}$ be a family of distributions on $\R$ with a location parameter $\theta$ (i.e., $P_\theta$ is the distribution of $Y+\theta$, where $Y\sim P_0$). Suppose there are $K$ arms, each arm $k\in [K]$ producing i.i.d.\ rewards $Y_{1,k},Y_{2,k},Y_{3,k}, \ldots \in\R$ with distribution $P_{\theta_k}$, for some $\theta_k\in\R$. At each round $n=1,2,\ldots$, the player chooses an arm $I_n\in[K]$ and receives the corresponding reward
$Y_{T_{I_n}(n-1),I_n}$, where $T_k(n-1) = \mathds 1(I_1=k)+\ldots+\mathds 1(I_{n-1}=k)$ is the number of times the arm $k$ was pulled during the rounds $1,\ldots,n-1$.
We let $\phi:\R\times \R\to\R$ be of the form $\phi(y,\theta)=\tilde\phi(y-\theta)$ and we assume that $0$ is the minimizer of $\E[\phi(Y-\theta)], \theta\in\R$, where $Y\sim P_0$. Therefore, for each arm $k\in [K]$, $\theta_k$ coincides with the population counterpart of the $M$-estimator defined in \Cref{sec:2}. In the rest of this section, we let Assumptions \ref{as:finite_expectation}, \ref{as:Phi_strongly_convex} and \ref{as:SG} hold for $P_0$, which implies that they automatically hold for each $P_{\theta}, \theta\in\R$. 
For every arm $k\in [K]$ and every sample size $n\geq 1$, we let $\hat\theta_{k,n}$ be a minimizer over $\theta\in\R$ of
$\frac{1}{n} \sum_{i=1}^n \phi(Y_{i,k},\theta)$. With this
notation, after $n$ rounds, we are able to compute the quantities
$\hat\theta_{k,T_k(n)}$ for $k\in[K]$. These quantities, 
combined with the confidence bounds furnished by the LIL of \Cref{MainThm1d}, lead to Robust lil'UCB algorithm 
described in \Cref{algo}\footnote{$\lambda$,
$\gamma$ and $n_0$ should be seen as tuning parameters for 
which our theoretical results give some guidance.}. 

\begin{center}\boxed{
	\begin{minipage}{0.80\linewidth}\linespread{1.5}
			\begin{algorithm}[H] 
			\SetKwInput{Input}{input}\SetKwInput{Output}{output}\SetKwInput{Init}{initialization}
					\caption{M-estimator lil'UCB. }
					\label{algo}\linespread{1.5}
						\Input{Confidence $\nu > 0$, parameters $\lambda, 
						\gamma > 0$, $n_0 \in \mathbb{N}$}
						\Init{Sample each arm $n_0$ times and set $n \leftarrow 
						K n_0$}
						Set $\delta = ((\sqrt{11\nu+9}-3)/11)^2$\\
						\For{$k$ in $1:K$}{Set $T_k(n) \leftarrow n_0$}
						\While{
						$\max_{k\in[K]} \big(T_k(n) - \lambda \sum_{\ell \neq k}T_{\ell}(n) \big)< 1$
						}
						{
						Sample arm
						${\DS
							I_n \leftarrow \arg\max_{k \in [K]}} \Big[ \hat{\theta}_{k, T_k(n)} + 
							\gamma\sqrt{\frac{\ln\ln 2T_k(n) +
							0.72\ln(\nicefrac{10.4}{\delta})}{T_k(n)}}
							\Big]
						$\\
						\For{$k$ in $1:K$}
						{\eIf{$I_n = k$}{$T_k(n+1) \leftarrow T_k(n) + 1$}
						{$T_k(n+1) \leftarrow T_k(n)$}}
						$n \leftarrow n+1$
						}		\Output{$\arg \max_{k \in [K]} T_k(n)$.\\ \hphantom{[7pt]}
						}
			\end{algorithm}
	\end{minipage}}
\end{center}

To state the theoretical results, let $k^*=
\operatorname{argmax}_{k\in [K]}\, \theta_k$ be the subscript corresponding to the best arm. We assume $k^*$ to be unique, and for $k\neq k^*$, define the sub-optimality gaps $\Delta_k=\theta_{k^*}-\theta_k$. 
We also introduce the quantities 
\begin{align}
	\mathbf{H}_1=\sum_{k\neq k^*}\frac{1}{\Delta_k^2} \quad\mbox{ and }\quad  
	\mathbf{H}_2=\sum_{k\neq k^*}\frac{\ln\ln(c/\Delta_k^2)}{\Delta_k^2},
\end{align}
where $c>e^2 \max_{k\in [K]}\Delta_k^2$ is a constant that appears in mathematical 
derivations.
\begin{theorem} \label{thm:UBforBandits}
For any $\nu\in (0,1)$ and  $\beta\in (0,2/(\sqrt{2}-1))$,
there exist positive constant $\lambda$, $C_1$, $C_2$ such 
that with probability at least $1-\nu$, \Cref{algo} used 
with parameters $\nu$, $\lambda$, $\gamma = 3.4(1+\beta)
\sigma/\alpha$ and $n_0$ stops after at most $Kn_0+C_1
\mathbf{H}_1+C_2 \mathbf{H}_2$ steps and returns the best 
arm.
\end{theorem}
The proof of this theorem, building on the proof of 
\cite[Theorem 2]{jamieson2014lil} is provided in the
supplementary material. Note that the order of magnitude 
of the number of steps, $O(\mathbf{H}_1+\mathbf{H}_2)$, 
is optimal, as demonstrated by the following result.

\begin{theorem} \label{thm:LBforBandits}
	Consider the RBAI framework with fixed confidence $\delta\in (0,1/2)$ described above and assume $K=2$. Let $\theta_1,\theta_2\in\R$ with $\theta_1\neq\theta_2$. Let $\tilde\phi$ be symmetric and the arm distributions be $\mathcal N(\theta_1,1)$ and $\mathcal N(\theta_2,1)$. Then, the gap between the two arms is given by $\Delta=|\theta_1-\theta_2|$ and any algorithm that finds the best of the two arms with probability at least $1-\delta$, for all values of $\Delta>0$, must satisfy 
\begin{align}	
	\limsup_{\Delta\to 0}\frac{\E[T]}{\Delta^{-2}\ln\ln(\Delta^{-2})} \geq 2-4\delta.
\end{align}	
\end{theorem}

To complete this section, we report the results 
of some  basic numerical experiments. 

\paragraph{Numerical experiments}

The values of $\theta_k$'s in our experiments 
were chosen according to the "$\alpha$-model" from \citep{jamieson2014lil} with $\alpha=0.3$. It imposes an exponential decrease on the means, that is $\theta_k = 1 - \left({k}/{K}\right)^\alpha$. Along with these means, we consider three reward generating processes : \emph{Gaussian rewards}, where $Y_{i,k}\simiid\nolinebreak\mathcal{N}(\theta_k, \sigma^2)$, \emph{Huber contaminated rewards}, where $Y_{i,k}\simiid\nolinebreak (1-\varepsilon)\mathcal{N}(\theta_k, \sigma^2) + \varepsilon Cauchy(\theta_k)$ for $\varepsilon=5\%$ and finally \emph{Student rewards}, where $Y_{i,k}\simiid\nolinebreak\mathcal Student_2(\theta_k)$ (i.e. Student distribution with $2$ degrees of freedom).
Note that all of these processes are mean and median centered around the $\theta_k$'s. To test the robustness of the compared algorithm, we tuned their parameters to fit the Gaussian reward scenario. 

In this set-up, we compared the original lil'UCB algorithm from \citep{jamieson2014lil}---see also \citep{JamiesonN14} for a more comprehensive experimental evaluation---and our version described
in \Cref{algo} where $\hat{\theta}_{k, n}$ is the empirical median of rewards from arm $k$ up to time $n$ (this corresponds to the $M$-estimator associated with the absolute loss). In order to lead a fair comparison  we assigned the same values to parameters shared by both procedures and set the values as in \citep{jamieson2014lil} : $\beta = 1$, $\lambda = (1 + \nicefrac{2}{\beta})^2$, $\sigma=0.5$, $\varepsilon=0.01$ and confidence $\nu=0.1$. Note that, as underlined by the authors of the paper, the choice of $\lambda$ does not fit the theoretical result from \citep{jamieson2014lil}. This choice is justified by the fact that $\lambda$ should theoretically be proportional to $(1 + \nicefrac{2}{\beta})^2$ with a constant converging to $1$ when the confidence approaches $0$. For our algorithm we chose $r=0.5$  which implies $\alpha=0.97, n_0=423$. The confidence level of our procedure is set to $\delta = \left(\nicefrac{\sqrt{11\nu + 9}-3}{11}\right)^2$ to get a global confidence level of $1-\nu$ .

The results, obtained by 200 independent runs of each algorithm on both settings, over several number of arms values, are depicted in \Cref{fig:nb_pulls} and \Cref{tab:correct_bai}.
The confidence of each procedure was adapted to reach a global confidence at least $90\%$. \Cref{tab:correct_bai} shows the proportion of times that each algorithm returned the correct best arm. We observe, that lil'UCB performed poorly on the non-Gaussian models. The performance of lil'UCB  deteriorates as the number of arms grow in the Huber scenario while it does not seem to be affected by the number of arms in the Student scenario.  In contrast, median lil'UCB performs well in all three scenarios.

\Cref{fig:nb_pulls}
displays the number of pulls for each algorithm when reaching its stopping criterion as a function of the number of arms $K$.
The curves represent the average number of pulls over the 200 runs while the colored areas around the curves are delimited by the maximum and the minimum  number of pulls over the 200 runs. We observe that the number of pulls of lil'UCB increases for non-Gaussian models and that the curves for median lil'UCB are almost identical for the three scenarios. The number of pulls for median lil'UCB is higher than the number of pulls for lil'UCB in the Gaussian and Student models. However, in the Huber model lil'UCB requires more pulls when the number of arms is higher. Note that the lil'UCB curve in the Gaussian model and the three median lil'UCB curves  have the same shape hence the same dependence in the problem complexity $\mathbf{H}_1$. 

\begin{figure}
    \centering
    \includegraphics[width=0.7\linewidth]{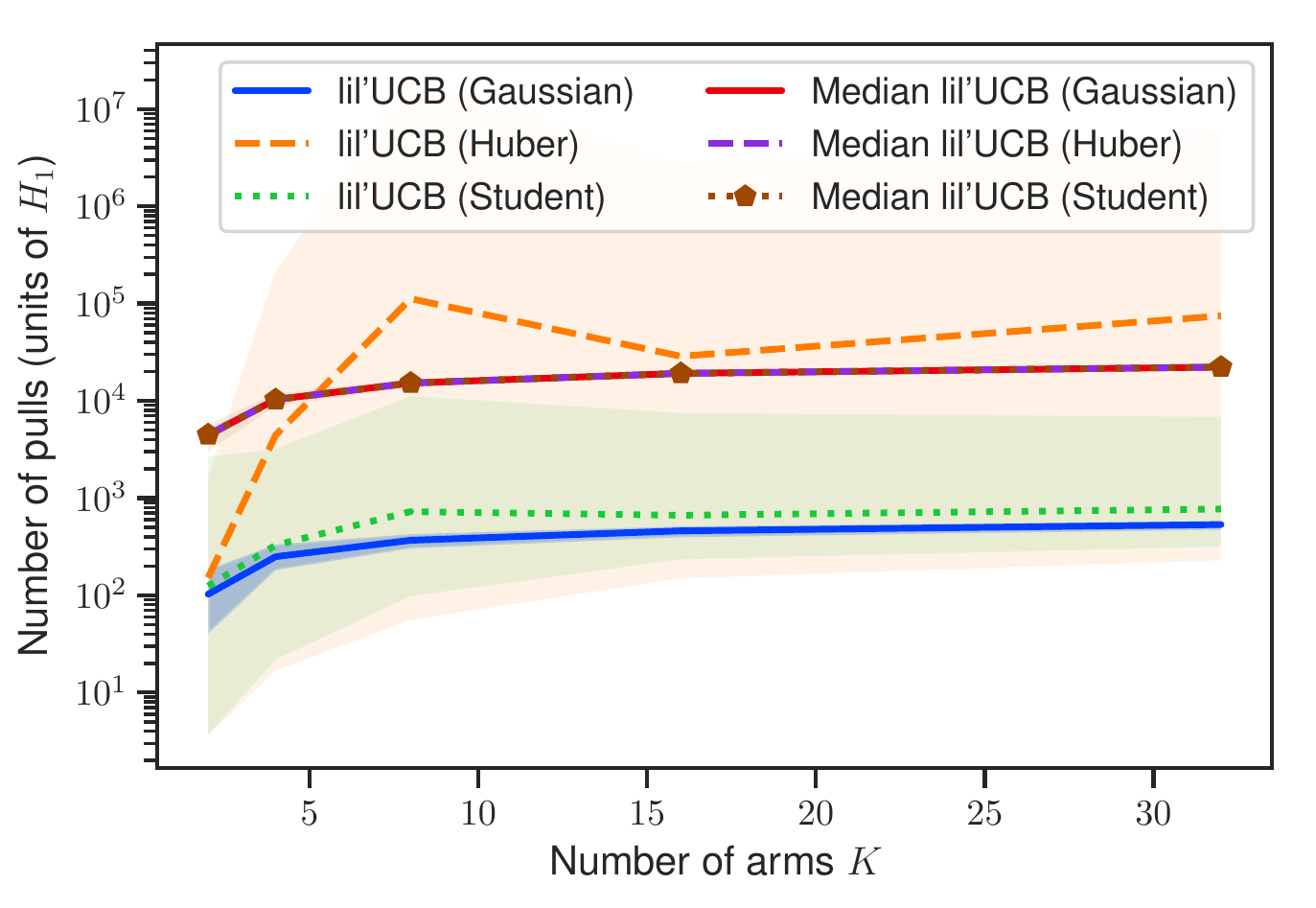}
    \caption{Total number of pulls in units of the complexity $H_1 \approx 3/2n $.}
    \label{fig:nb_pulls}
\end{figure}

These basic numerical experiments illustrate the lack of robustness of lil'UCB against heavy tail scenario and the effective robustness of median lil'UCB. However, this robustness comes with a higher number of pulls which is superfluous in sub-Gaussian scenario. Therefore median lil'UCB should be preferred to vanilla lil'UCB only if one suspects heavy-tailed rewards.

\begin{table}
    \centering
    \caption{Proportion of correct best arm identification (over 200 runs per scenario/algorithm).}
\begin{tabular}{llrrrrr}
\toprule
Scenario & Algorithm &   K=2 &    K=4 &    K=8 &   K=16 &   K=32 \\
\midrule
\multirow{2}{*}{Gaussian} & lil'UCB &  1.000 &  1.000 &  1.000 &  1.000 &  1.000 \\
        & Median lil'UCB &  1.000 &  1.000 &  1.000 &  1.000 &  1.000 \\
\midrule
\multirow{2}{*}{Huber} & lil'UCB &  0.915 &  0.820 &  0.750 &  0.745 &  0.645 \\
        & Median lil'UCB &  1.000 &  1.000 &  1.000 &  1.000 &  1.000 \\
\midrule
\multirow{2}{*}{Student} & lil'UCB &  0.915 &  0.975 &  0.915 &  0.965 &  0.950 \\
        & Median lil'UCB &  1.000 &  1.000 &  1.000 &  1.000 &  1.000 \\
\bottomrule
\end{tabular}
    \label{tab:correct_bai}
\end{table}
\label{sec:4}

\section{Conclusion and further work}\label{sec:5}

We have proved nonasymptotic law of iterated
logarithm for general $M$-estimators both in one
dimensional and in multidimensional setting. 
These results can be seen as off-the-shelf 
deviation bounds  that are uniform in the sample 
size and, therefore, suitable for on-line learning 
problems and problems in which the sample size may
depend on the observations. There are several avenues 
for future work. For simplicity, in the multi-dimensional 
case, the population risk is assumed to be above 
an elliptic paraboloid on the whole space. First 
in our agenda is to replace this condition by a local 
curvature one. A second interesting line of future 
research is to prove the LIL for sequential estimators 
such as the on-line gradient descent. It would also be
of interest to obtain ``in-expectation'' bounds of the
same type as those established for the mean in
\citep{Shin19}. Regarding applications, the 
multi-dimensional LIL could be used to obtain theoretical 
guarantees in bandit problems with covariates. 

\bibliography{biblio}

\begin{thebibliography}{32}
\providecommand{\natexlab}[1]{#1}
\providecommand{\url}[1]{\texttt{#1}}
\expandafter\ifx\csname urlstyle\endcsname\relax
  \providecommand{\doi}[1]{doi: #1}\else
  \providecommand{\doi}{doi: \begingroup \urlstyle{rm}\Url}\fi

\bibitem[{Altschuler} et~al.(2018){Altschuler}, {Brunel}, and {Malek}]{Brunel}
J.~{Altschuler}, V.-E. {Brunel}, and A.~{Malek}.
\newblock {Best Arm Identification for Contaminated Bandits}.
\newblock \emph{arXiv e-prints}, art. arXiv:1802.09514, Feb. 2018.

\bibitem[Arcones(1994)]{arcones1994some}
M.~A. Arcones.
\newblock Some strong limit theorems for m-estimators.
\newblock \emph{Stochastic Processes and Their Applications}, 53\penalty0
  (2):\penalty0 241--268, 1994.

\bibitem[Audibert et~al.(2010)Audibert, Bubeck, and Munos]{audibert2010best}
J.~Audibert, S.~Bubeck, and R.~Munos.
\newblock Best arm identification in multi-armed bandits.
\newblock In \emph{{COLT} 2010 - The 23rd Conference on Learning Theory, Haifa,
  Israel, June 27-29, 2010}, pages 41--53, 2010.

\bibitem[Boucheron et~al.(2013)Boucheron, Lugosi, and
  Massart]{boucheron2013concentration}
S.~Boucheron, G.~Lugosi, and P.~Massart.
\newblock \emph{Concentration inequalities: A nonasymptotic theory of
  independence}.
\newblock Oxford university press, 2013.

\bibitem[Chernozhukov(2005)]{Chernozhukov05}
V.~Chernozhukov.
\newblock Extremal quantile regression.
\newblock \emph{Ann. Statist.}, 33\penalty0 (2):\penalty0 806--839, 2005.

\bibitem[Collins(1977)]{Collins}
J.~R. Collins.
\newblock Upper bounds on asymptotic variances of {$M$}-estimators of location.
\newblock \emph{Ann. Statist.}, 5\penalty0 (4):\penalty0 646--657, 1977.

\bibitem[Farrell(1964)]{farrell1964asymptotic}
R.~H. Farrell.
\newblock Asymptotic behavior of expected sample size in certain one sided
  tests.
\newblock \emph{The Annals of Mathematical Statistics}, pages 36--72, 1964.

\bibitem[Gabillon et~al.(2012)Gabillon, Ghavamzadeh, and Lazaric]{GabillonGL12}
V.~Gabillon, M.~Ghavamzadeh, and A.~Lazaric.
\newblock Best arm identification: {A} unified approach to fixed budget and
  fixed confidence.
\newblock In \emph{Advances in Neural Information Processing Systems 25: 26th
  Annual Conference on Neural Information Processing Systems 2012. Proceedings
  of a meeting held December 3-6, 2012, Lake Tahoe, Nevada, United States.},
  pages 3221--3229, 2012.

\bibitem[He and Wang(1995)]{he1995law}
X.~He and G.~Wang.
\newblock Law of the iterated logarithm and invariance principle for
  m-estimators.
\newblock \emph{Proceedings of the American Mathematical Society}, 123\penalty0
  (2):\penalty0 563--573, 1995.

\bibitem[Hoerl and Kennard(2000)]{Ridge}
A.~E. Hoerl and R.~W. Kennard.
\newblock Ridge regression: Biased estimation for nonorthogonal problems.
\newblock \emph{Technometrics}, 42\penalty0 (1):\penalty0 80--86, 2000.

\bibitem[Howard et~al.(2018)Howard, Ramdas, McAuliffe, and
  Sekhon]{howard2018uniform}
S.~R. Howard, A.~Ramdas, J.~McAuliffe, and J.~Sekhon.
\newblock Uniform, nonparametric, non-asymptotic confidence sequences.
\newblock \emph{arXiv preprint arXiv:1810.08240}, 2018.

\bibitem[Hsu and Sabato(2016)]{Hsu}
D.~Hsu and S.~Sabato.
\newblock Loss minimization and parameter estimation with heavy tails.
\newblock \emph{Journal of Machine Learning Research}, 17\penalty0
  (18):\penalty0 1--40, 2016.

\bibitem[Huber and Ronchetti(2009)]{Huber_book}
P.~J. Huber and E.~M. Ronchetti.
\newblock \emph{Robust statistics}.
\newblock Wiley Series in Probability and Statistics. John Wiley \& Sons, Inc.,
  Hoboken, NJ, second edition, 2009.

\bibitem[Huber et~al.(1964)]{huber1964robust}
P.~J. Huber et~al.
\newblock Robust estimation of a location parameter.
\newblock \emph{The annals of mathematical statistics}, 35\penalty0
  (1):\penalty0 73--101, 1964.

\bibitem[Jamieson et~al.(2014)Jamieson, Malloy, Nowak, and
  Bubeck]{jamieson2014lil}
K.~Jamieson, M.~Malloy, R.~Nowak, and S.~Bubeck.
\newblock lil’ucb: An optimal exploration algorithm for multi-armed bandits.
\newblock In \emph{Conference on Learning Theory}, pages 423--439, 2014.

\bibitem[Jamieson and Nowak(2014)]{JamiesonN14}
K.~G. Jamieson and R.~D. Nowak.
\newblock Best-arm identification algorithms for multi-armed bandits in the
  fixed confidence setting.
\newblock In \emph{48th Annual Conference on Information Sciences and Systems,
  {CISS} 2014, Princeton, NJ, USA, March 19-21, 2014}, pages 1--6, 2014.

\bibitem[Kaufmann and Koolen(2018)]{Kaufmann_Koolen}
E.~Kaufmann and W.~M. Koolen.
\newblock Mixture martingales revisited with applications to sequential tests
  and confidence intervals.
\newblock \emph{CoRR}, abs/1811.11419, 2018.

\bibitem[Kaufmann et~al.(2016)Kaufmann, Capp{\'{e}}, and Garivier]{Kaufman16}
E.~Kaufmann, O.~Capp{\'{e}}, and A.~Garivier.
\newblock On the complexity of best-arm identification in multi-armed bandit
  models.
\newblock \emph{Journal of Machine Learning Research}, 17:\penalty0 1:1--1:42,
  2016.
\newblock URL \url{http://jmlr.org/papers/v17/kaufman16a.html}.

\bibitem[Khintchine(1924)]{Khintchine1924}
A.~Khintchine.
\newblock {\"U}ber einen satz der wahrscheinlichkeitsrechnung.
\newblock \emph{Fundamenta Mathematicae}, 6\penalty0 (1):\penalty0 9--20, 1924.

\bibitem[Kolmogoroff(1929)]{Kolmogoroff1929}
A.~Kolmogoroff.
\newblock Über das gesetz des iterierten logarithmus.
\newblock \emph{Mathematische Annalen}, 101:\penalty0 126--135, 1929.

\bibitem[Koltchinskii(2011)]{koltchinskii}
V.~Koltchinskii.
\newblock \emph{Oracle inequalities in empirical risk minimization and sparse
  recovery problems}, volume 2033 of \emph{Lecture Notes in Mathematics}.
\newblock Springer, Heidelberg, 2011.
\newblock Lectures from the 38th Probability Summer School held in Saint-Flour,
  2008.

\bibitem[Lecu{\'e} and Rigollet(2014)]{lecue2014}
G.~Lecu{\'e} and P.~Rigollet.
\newblock Optimal learning with {$Q$}-aggregation.
\newblock \emph{Ann. Statist.}, 42\penalty0 (1):\penalty0 211--224, 02 2014.

\bibitem[Loh(2017)]{loh2017}
P.-L. Loh.
\newblock Statistical consistency and asymptotic normality for high-dimensional
  robust {$M$}-estimators.
\newblock \emph{Ann. Statist.}, 45\penalty0 (2):\penalty0 866--896, 04 2017.

\bibitem[Maillard(2019)]{pmlr-v98-maillard19a}
O.-A. Maillard.
\newblock Sequential change-point detection: Laplace concentration of scan
  statistics and non-asymptotic delay bounds.
\newblock In A.~Garivier and S.~Kale, editors, \emph{Proceedings of the 30th
  International Conference on Algorithmic Learning Theory}, volume~98 of
  \emph{Proceedings of Machine Learning Research}, pages 610--632, Chicago,
  Illinois, 22--24 Mar 2019. PMLR.
\newblock URL \url{http://proceedings.mlr.press/v98/maillard19a.html}.

\bibitem[Maronna(1976)]{Maronna}
R.~A. Maronna.
\newblock Robust m-estimators of multivariate location and scatter.
\newblock \emph{The Annals of Statistics}, 4\penalty0 (1):\penalty0 51--67,
  1976.

\bibitem[Negahban et~al.(2012)Negahban, Ravikumar, Wainwright, and
  Yu]{negahban2012}
S.~N. Negahban, P.~Ravikumar, M.~J. Wainwright, and B.~Yu.
\newblock A unified framework for high-dimensional analysis of $m$-estimators
  with decomposable regularizers.
\newblock \emph{Statist. Sci.}, 27\penalty0 (4):\penalty0 538--557, 11 2012.

\bibitem[Portnoy(1984)]{Portnoy}
S.~Portnoy.
\newblock Asymptotic behavior of {$M$}-estimators of {$p$} regression
  parameters when {$p^{2}/n$} is large. {I}. {C}onsistency.
\newblock \emph{Ann. Statist.}, 12\penalty0 (4):\penalty0 1298--1309, 1984.

\bibitem[Rakhlin et~al.(2012)Rakhlin, Shamir, and Sridharan]{Rakhlin2012}
A.~Rakhlin, O.~Shamir, and K.~Sridharan.
\newblock Making gradient descent optimal for strongly convex stochastic
  optimization.
\newblock In \emph{{ICML} 2012, Edinburgh, Scotland, UK, June 26 - July 1,
  2012}. icml.cc / Omnipress, 2012.

\bibitem[Shin et~al.(2019)Shin, Ramdas, and Rinaldo]{Shin19}
J.~Shin, A.~Ramdas, and A.~Rinaldo.
\newblock On the bias, risk and consistency of sample means in multi-armed
  bandits.
\newblock \emph{CoRR}, abs/1902.00746, 2019.
\newblock URL \url{http://arxiv.org/abs/1902.00746}.

\bibitem[Sridharan et~al.(2009)Sridharan, Shalev-shwartz, and Srebro]{Karthik}
K.~Sridharan, S.~Shalev-shwartz, and N.~Srebro.
\newblock Fast rates for regularized objectives.
\newblock In \emph{Advances in Neural Information Processing Systems 21}, pages
  1545--1552. Curran Associates, Inc., 2009.

\bibitem[van~der Vaart(1998)]{vanderVaart}
A.~W. van~der Vaart.
\newblock \emph{Asymptotic statistics}, volume~3 of \emph{Cambridge Series in
  Statistical and Probabilistic Mathematics}.
\newblock Cambridge University Press, Cambridge, 1998.

\bibitem[Wang et~al.(2014)Wang, Wonka, and Ye]{WangWY14}
J.~Wang, P.~Wonka, and J.~Ye.
\newblock Scaling {SVM} and least absolute deviations via exact data reduction.
\newblock In \emph{{ICML} 2014, Beijing, China, 21-26 June 2014}, volume~32 of
  \emph{{JMLR} Workshop and Conference Proceedings}, pages 523--531. JMLR.org,
  2014.

\end{thebibliography}

\newpage
\section{Proofs}\label{sec:6}

This section contains the proofs of the main theorems stated and discussed
in the main body of the paper. Some technical lemmas used in the proofs
of this section are postponed to \Cref{sec:postponed_proofs}.

\subsection{Proof of \texorpdfstring{\Cref{MainThm1d}}{Theorem 1}}

Let $\delta \in (0, 1)$. Define the sequence $t(n)$ by setting 
\begin{align}\label{t_n}
    t(n) = \frac{3.4 \sigma}{\alpha}\sqrt{\frac{\ln \ln 2n + 0.72\ln (\nicefrac{10.4}{\delta})}{n}}
\end{align} 
for any integer $n \geq 1$ and define $n_0=n_0(\alpha,r,\delta)$
to be the smallest integer $n\ge 1$ for which $t(n) \le r$. We 
intentionally omit the dependence of $t(n)$ in $\delta$ to 
lighten notations. We only detail the proof for the upper 
bound of the probability of the event 
\begin{align}
    \mathcal{A} \coloneqq \left\{ \exists n \geq n_0 
    \text{ such that } \widehat{\theta}_n - \theta^* > t(n) \right\},    
\end{align}
the proof for upper bounding the probability of the event 
$\mathcal{A}' \coloneqq \big\{ \exists n \geq n_0, 
\theta^* - \hat{\theta}_n > t(n) \big\}$ is very similar. 
Our proof can be decomposed into two steps : first, we show 
that we can reduce the problem of upper bounding the probability 
of the event $\mathcal{A}$ to the problem of uniformly bounding 
a sum of sub-Gaussian random variables ; then we employ a tight 
uniform concentration inequality for the sum of sub-Gaussian 
random variables.
\begin{lemma}\label{lem:m_estimator_to_SG}
Under \Cref{as:convex_phi,as:Phi_strongly_convex,as:SG}, for 
any integer $n \geq n_0$ and positive real $t \in (0, r]$, 
there exist $n$ Ni.i.d.\ $\sigma^2$-sub-Gaussian random variables
$Z_1(t),\ldots,Z_n(t)$ such that 
\begin{align}
    \mathcal{A}_n(t) \coloneqq \big\{\hat{\theta}_n > 
    \theta^* + t \big\}\subset \mathcal{B}_n (t) = 
    \bigg\{\sum_{i=1}^n Z_i(t) \geq \frac{\alpha}{2} 
    n t\bigg\}.
\end{align}
\end{lemma}
\begin{myproof}
For any integer $n\geq n_0$ and real $t\in (0,r]$, we set
\begin{align}
    S_n(t) &= n\big(\hat\Phi_n(\theta^*)-\Phi(\theta^*)\big)- n\big(\hat\Phi_n(\theta^*+t)-\Phi(\theta^*+t)\big)\\
    &= n\big(\hat\Phi_n(\theta^*)-\hat\Phi_n(\theta^*+t)\big)+ n\big(\Phi(\theta^*+t)-\Phi(\theta^*)\big).\label{eqq1}
\end{align} 
\Cref{as:convex_phi} ensures that the empirical risk 
$\hat\Phi_n$ is convex and coercive, thus, 
\begin{align}
    \mathcal{A}_n(t) \subset \big\{\hat\Phi_n(\theta^*)\geq 
\hat\Phi_n(\theta^*+t)\big\},
\end{align} 
see \Cref{fig:U_shape} for an illustration of this implication. 
Using \eqref{eqq1} and the lower-boundedness of the population risk 
$\Phi$ by a quadratic function (\Cref{as:Phi_strongly_convex}), 
we arrive at
\begin{align}
    \mathcal{A}_n(t) \subset \big\{S_n(t)  \geq 
    n\left(\Phi(\theta^* + t) - \Phi(\theta^*)\right)\big\}
    \subset \big\{S_n(t)  \geq \frac{\alpha}{2}nt^2\big\}
    \subset\bigg\{\frac{S_n(t)}{t} \geq \frac{\alpha}{2}nt^2\bigg\}.
\end{align}
Finally, using the definition of $\hat\Phi_n$, we can write 
$S_n(t)$ as follows
\begin{align}
    \frac{S_n(t)}{t} = \sum_{i=1}^n t^{-1}\big\{\underbrace{
    \phi(Y_i,\theta^*)
    -\phi(Y_i,\theta^*+t)-\mathbb E\big[\phi(Y_i,\theta^*)-
    \phi(Y_i,\theta^*+t)\big]}_{:=tZ_i(t)}\big\}. 
\end{align}
The random variables $Z_i(t)$ are clearly centered and i.i.d.
Furthermore, it follows from \Cref{as:SG} that 
$Z_i(t)$ is sub-Gaussian variables with variance proxy $\sigma^2$. 
This completes the proof. 
\end{myproof}
\definecolor{xdxdff}{rgb}{0.,0.,1}
\begin{figure}
\centering
\begin{tikzpicture}[line cap=round,line join=round,>=triangle 45,x=1cm,y=1cm]
\begin{axis}[
x=1cm,y=0.5cm,
axis lines=middle,
xlabel = $\theta$,
grid style=dashed,
ymajorgrids=false,
xmajorgrids=false,
xmin=-0.5,
xmax=7.435978260869557,
ymin=-0.5,
ymax=7.539371980676332,
xtick={\empty},
ytick={\empty},
extra x ticks={2, 3,4},
extra x tick labels={$\theta^*$, $\theta^* + t$, $\hat{\theta}_{n}$}],
\clip(-1.0509782608695646,-0.6747584541062818) rectangle (12.435978260869557,7.539371980676332);
\draw [samples=1000,rotate around={0:(4,1)},xshift=4cm,yshift=0.5cm,line width=2pt,domain=-5:5)] plot (\x,{(\x)^2/2/0.5});
\begin{scriptsize}
\draw[color=black] (6.7,6.1410024154589373) node {$\hat{\Phi}_n$};
\draw [fill=xdxdff] (2,5) circle (2pt);
\draw[color=xdxdff] (1.3,5) node {$\hat{\Phi}_n(\theta^*)$};
\draw [fill=xdxdff] (3,2) circle (2pt);
\draw[color=xdxdff] (1.9,2) node {$\hat{\Phi}_n(\theta^* + t)$};
\draw [fill=xdxdff] (4,1) circle (2pt);
\draw[color=xdxdff] (4.05,1.9) node {$\hat{\Phi}_n(\hat{\theta}_{n})$};
\end{scriptsize}
\end{axis}
\end{tikzpicture}
\caption{Illustration of the shape of the function $\hat{\Phi}_n$.}
\label{fig:U_shape}
\end{figure}
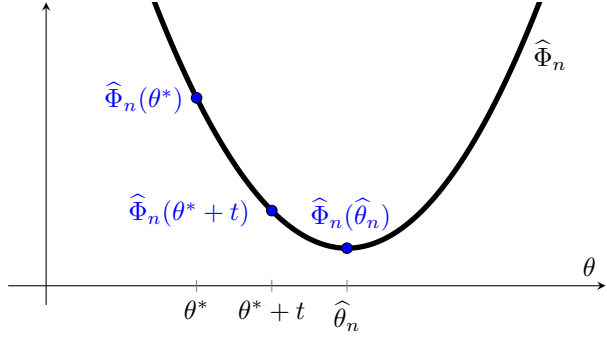

\Cref{lem:m_estimator_to_SG} tells us that, in order to bound the probability of the event 
\begin{align}
    \mathcal{A} = \bigcup_{n=n_0}^\infty \mathcal{A}_n\big(t(n)\big)
\end{align}
it suffices to bound the probability of the event 
\begin{align}
    \mathcal{B} \coloneqq \bigcup_{n=n_0}^\infty \mathcal{B}_n
    \big(t(n)\big)= \bigg\{ \exists n \geq n_0\text{ such that } 
    \sum_{i=1}^n Z_i\big(t(n)\big) \geq \frac{\alpha}{2} n t(n) \bigg\}.    
\end{align} 
Thus, we need a uniform in sample size upper bound on the sum 
of sub-Gaussian random variables. We will use a special case 
of \cite[Theorem 1]{howard2018uniform} which we now state 
(see Eq.~(7) in the original paper).
\begin{theorem}[{\citealp[Theorem 1]{howard2018uniform}}]
\label{thm:Howard}
Let $Z_1, Z_2, \dots$ be independent, zero-mean, $\sigma^2$-sub-Gaussian 
random variables. It holds that, for any confidence $\delta \in (0, 1)$,
\begin{align}
    \mathbb{P}\bigg( \exists\, n \geq 1 : \sum_{i=1}^n Z_i \geq 1.7 \sigma \sqrt{n\big( \ln \ln (2n) + 0.72 \ln({5.2}/{\delta}) \big)} \bigg) \leq \delta.
\end{align}
\end{theorem}
Combining \Cref{lem:m_estimator_to_SG} with \Cref{thm:Howard}, and
taking into account the definition \eqref{t_n} of $t(n)$, we get
\begin{align}
    \mathbb{P}\left(\mathcal{A}\right) \leq \mathbb{P}
    \bigg(\exists n \geq n_0 \text{ such that } \sum_{i=1}^n 
    Z_i\big(t(n)\big) \geq \frac{\alpha}{2} n t(n) \bigg) 
    \leq \delta/2.
\end{align}
One can easily check that an identical upper bound for the probability 
of the event 
\begin{align}
    \mathcal{A}' = \left\{ \exists n \geq n_0 
    \text{ such that }\theta^* - \hat{\theta}_n > t(n) \right\}
\end{align}
can be obtained using the same arguments.

\begin{remark}
Several uniform bounds on the sum of sub-Gaussian random variables 
have been proved (see, e.g. \citep{jamieson2014lil,pmlr-v98-maillard19a} 
and the other theorems from \citep{howard2018uniform}). 
\Cref{fig:comparison} and \Cref{tab:sum_SG_upper_bound} shows a 
comparison between those bounds. The bound from \citep{jamieson2014lil} 
is loosest for any sample size. The bound from \citep{pmlr-v98-maillard19a} 
is the tightest for small sample size while the one from 
\cite{howard2018uniform} becomes the tightest when the sample 
size increases. 
\end{remark}

\begin{table}
\caption{Uniform upper bounds for sum of $t$ i.i.d. 1-sub-Gaussian random variables.}
\label{tab:sum_SG_upper_bound}
\centering
\begin{tabular}{lll}
    \toprule
    Reference & Bound & Confidence \\
    \midrule
    \cite{jamieson2014lil} & $1.57\left[t \left(\ln\ln(1.01t) + \ln(\nicefrac{1}{\delta})\right)\right]^{\nicefrac{1}{2}}$ & $21154 \delta^{1.01}$\\
    \cite{howard2018uniform} & $1.44\left[t\left(1.4 \ln \ln(2t) + \ln\left(\nicefrac{5.19}{\delta}\right) \right)\right]^{\nicefrac{1}{2}}$ & $\delta$\\ 
    \cite{pmlr-v98-maillard19a} & $1.42 \left[ \left(t+1\right)\left(\ln(\sqrt{t+1}) + \ln(\nicefrac{1}{\delta})\right)\right]^{\nicefrac{1}{2}}$ & $\delta$\\
    \bottomrule
 \end{tabular}
\end{table}
\begin{figure}
    \centering
    \includegraphics[width=0.7\linewidth]{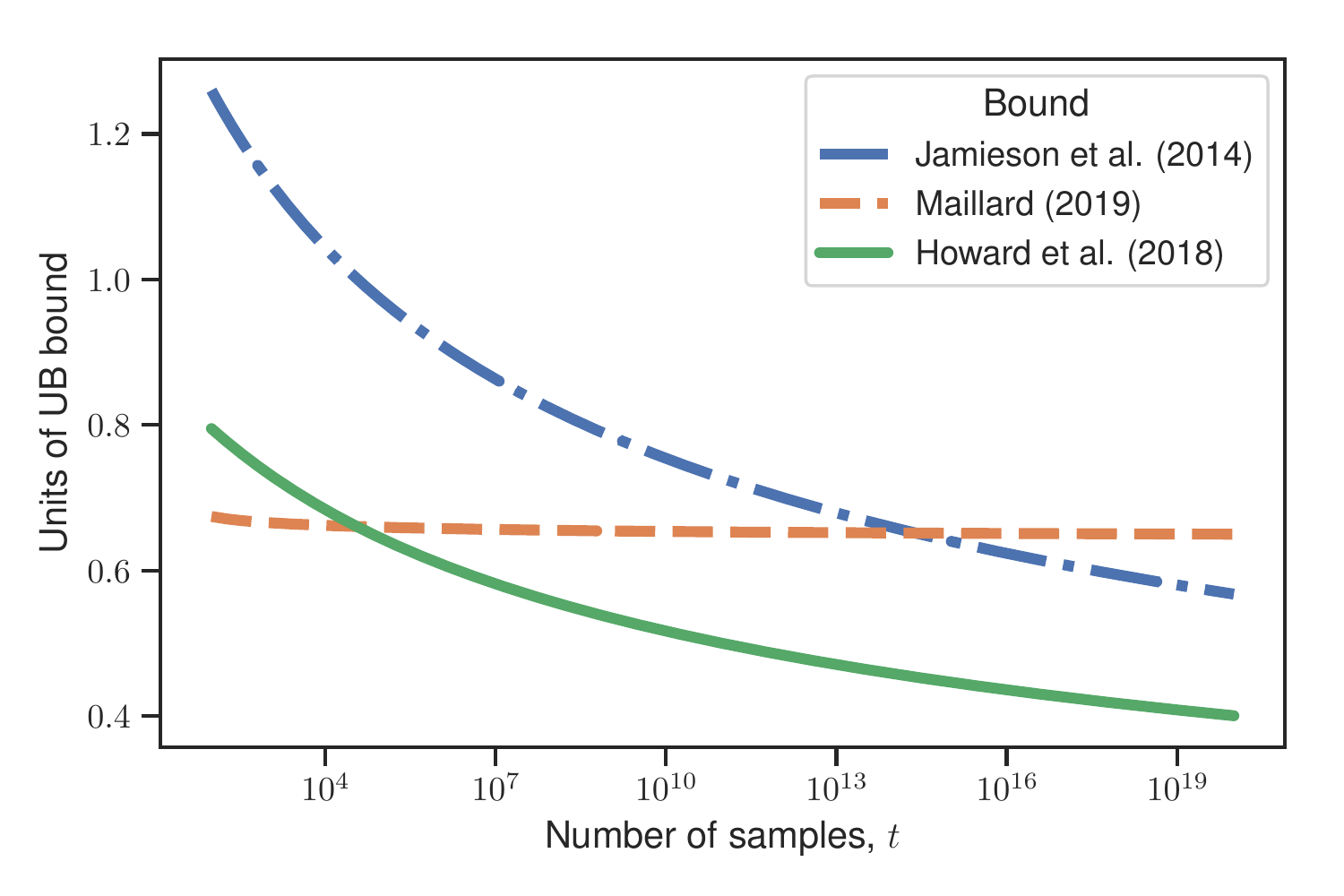}
    \caption{Comparison of uniform, high-probability, upper tail bounds for the sum of i.i.d. sub-Gaussian random variables scaled by $c_{n, \delta}^{UB}$ bound (see \Cref{sec:2}). \citet[Lemma 3]{jamieson2014lil} with $\varepsilon=0.02$, \citet[Lemma 15]{pmlr-v98-maillard19a} and \citet[Theorem 1]{howard2018uniform} with $ \eta=2.04, s=1.4$. Global confidence is set to $\nu = 0.1$.}
    \label{fig:comparison}
\end{figure}%

\subsection{Proof of \texorpdfstring{\Cref{main:thm2}}{Theorem 2}}

Without loss of generality, we assume hereafter that $\bs{a}$ is a unit vector. 
Let $\beta\in(1,2)$ and $\varepsilon>0$ be two constants that we will choose to be equal to $1.1$ and $0.2$, respectively. Throughout the proof we consider, for $k \in \mathbb{N}$, the sequence of integers, $n_0 = 4$, $n_{k+1} = \lceil\beta n_k\rceil$ and the sequence of integer intervals $I_k = [ n_k, n_{k+1})\cap\mathbb{N}$.
We define the sequence $(t(n))_{n \in \mathbb{N}}$ by setting
\begin{align}
        t(n) = \frac{10\varrho_n B}{\sqrt{3}}\sqrt{\frac{(1+\varepsilon)\ln\ln_\beta n + \ln(1/\delta) + 5/8}{n}}, \quad \text{ for } n \geq 1\label{cn}.
\end{align}
For readability we write $t(n_k) = t_{k}$ for any integer $k$. We wish to upper bound the probability of
the event
\begin{align}
    \mathcal{A} = \bigcup_{n=4}^\infty \mathcal{A}_n,\quad \text{where} \quad\mathcal{A}_n = \{ {\bs a}^\top(\btheta^* - \bm{\hat{\theta}}_n) > t(n)\}.
\end{align}
Define the set $\mathcal{V} = \big\{{\bs v} \in \mathbb{R}^d, {\bs v}^\top{\bs a} = 1\big\}$ and the random variable 
\begin{align}
    S_n({\bs w}) = n\left(\hat\Phi_n(\btheta^*) -\Phi_n(\btheta^*)\right) -n\left(\hat\Phi_n(\btheta^*-\bs{w}) - \Phi_n(\btheta^* - \bs w) \right).
\end{align}
We have the following lemma resulting from the convexity assumptions.
\begin{lem}\label{lem:1}
Under \Cref{as:convex_lipschitz_phi,as:convex_penalty,as:Phi_strongly_convex_everywhere}, for any integers $k  \in \mathbb{N}, n \in I_k$, the event $\mathcal{A}_n$ is included in the event 
\begin{align}
    \mathcal{B}_{n} \coloneqq
    \left\{ \sup_{{\bs w} \in t_{k+1}\mathcal{V}} \left[ S_n(\bs w) - (\nicefrac{ \alpha_n}{2}) n \| {\bs w}\|^2\right] \geq 0   \right\}.
\end{align}
\end{lem}
The proofs of the lemmas stated in this section are postponed to \Cref{sec:postponed_proofs}.
Combining \Cref{lem:1} with a union bound gives
\begin{align}
    \mathbb{P}\left(\mathcal{A} \right)
    &\le  \mathbb{P}\bigg(\bigcup_{k \geq 0} \bigcup_{n \in I_k}\mathcal{B}_{n} \bigg)\leq \sum_{k \geq 0} \mathbb{P}\bigg(\bigcup_{n \in I_k} \mathcal{B}_{n} \bigg).
\end{align}
Let k be an integer. Since the sequence $(\alpha_n)_n$ is non-increasing we have, for any integer $n \in I_k$, $\alpha_n \geq \alpha_{n_{k+1}}$. Setting $\beta = 1.1$ we have $n_k/n_{k+1} \geq 4/5$ for $n \geq 4$. Thus, for any positive real $\lambda$,
\begin{align}
    \mathbb{P}\bigg(\bigcup_{n \in I_k} \mathcal{B}_{n} \bigg) &\leq \mathbb{P}\bigg( \sup_{n \in I_k} \sup_{\bs w \in t_{k+1}\mathcal{V}} \bigg[S_n(\bs w) - \frac{\alpha_n}{2} n_k \lVert \bs w \rVert_2^2 \bigg] \geq 0 \bigg)\\
    &\leq  \mathbb{P}\bigg( \sup_{n \in I_k} \sup_{\bs w \in t_{k+1}\mathcal{V}} \bigg[S_n(\bs w) - \frac{2\alpha_{n_{k+1}}}{5} n_{k+1} \lVert \bs w \rVert_2^2 \bigg] \geq 0 \bigg)\\
    &\leq \mathbb{P}\bigg( \sup_{n \in I_k} \sup_{\bs w \in t_{k+1}\mathcal{V}} \exp\left\{\lambda \left(S_n(\bs w) - \frac{2\alpha_{n_{k+1}}}{5} n_{k+1} \lVert \bs w \rVert_2^2 \right) \right\} \geq 1 \bigg).
\end{align}
The stochastic process $\left(\sup_{\bs w \in t_{k+1}\mathcal{V}} \exp\left\{\lambda \left(S_n(\bs w) - 2\alpha_{n_{k+1}} n_{k+1} \lVert \bs w \rVert_2^2/5 \right) \right\}\right), n \in \mathbb{N}^*$, is a submartingale with respect to its natural filtration, therefore, Doob's maximal inequality for submartingales yields,
\begin{align}
    \mathbb{P}\bigg(\bigcup_{n \in I_k} \mathcal{B}_{n} \bigg)
    &\leq \inf_{\lambda > 0} \mathbb{E}\left[\sup_{\bs w \in t_{k+1}\mathcal{V}} \exp\left\{ \lambda \left(S_{n_{k+1}}(\bs w) - \frac{2\alpha_{n_{k+1}}}{5} n_{k+1} \lVert \bs w \rVert_2^2 \right)\right\} \right].\label{eq:expectation_exp_sum}
\end{align}
The next lemma uses classic tools from empirical processes theory such as the symmetrization trick and the contraction principle to bound the expectation from \eqref{eq:expectation_exp_sum}.
\begin{lem}\label{lem:symmetrization_contraction}
Under \Cref{as:convex_lipschitz_phi}, given a positive integer $m$ and three positive real numbers $t$, $\alpha$ and $\lambda$, letting $t' = (\nicefrac{2m \alpha }{L})t$, we have, 
\begin{align}
    \inf_{\lambda > 0} \mathbb{E}\left[\sup_{\bs{w} \in t \mathcal{V}} \exp\left\{ \lambda \left( S_m(\bs{w}) - \alpha m \lVert \bs{w} \rVert_2^2 \right) \right\} \right] \leq \inf_{\lambda > 0} \mathbb{E}\left[\sup_{\bs{w} \in t' \mathcal{V}} \exp\left\{ \lambda \left(\bs{w}^\top \mathbf{X} \bs{\varepsilon} - \lVert \bs{w} \rVert_2^2/2 \right) \right\} \right].
\end{align}
\end{lem}
Applying \Cref{lem:symmetrization_contraction} with $m=n_{k+1}, \alpha = 2\alpha_{n_{k+1}}/5$ and $t = t_{k+1}$ gives
\begin{align}
    \mathbb{P}\bigg(\bigcup_{n \in I_k} \mathcal{B}_{n} \bigg) \leq \inf_{\lambda > 0} \mathbb{E}\left[\sup_{\bs w \in s_{k+1} \mathcal{V}} \exp \left\{ \lambda(\bs w^\top \mathbf{X} \bs\varepsilon - \lVert \bs w \rVert_2^2/2) \right\} \right], \quad s_{k+1} = \frac{4n_{k+1}}{5\varrho_{n_{k+1}}}t_{k+1}.\label{eq:expectation_exp_rademacher}
\end{align}
For fixed $\mathbf{X}$ and $\bs\varepsilon$, define the concave quadratic function $G(\bs w) \coloneqq \bs w^\top \mathbf{X} \bs\varepsilon - \lVert \bs w \rVert_2^2/2 $. The next lemma results from explicitly computing the supremum inside the expectation in \eqref{eq:expectation_exp_rademacher} and bounding the resulting moment generating function.
For the next lemma, we denote by $B_{\bs{a}^\top \bs{X}}$ the smallest constant $B$ for which
$\mathbb{P}(|\bs{a}^\top\bs{X}_1|\le B)=1$. It is clear that $B_{\bs{a}^\top \bs{X}}\le B_{\|\bs{X}\|}$.
Nevertheless, we prefer to use the constant $B_{\bs{a}^\top \bs{X}}$ for this lemma in order to keep the inequality as tight as possible.
\begin{lem}\label{lem:first_expectation}
    Let $I$ be a finite set of cardinality $m \in \mathbb{N}$. Let $(\bX_i)_{i \in I}$ be i.i.d.\ random 
    vectors in $\mathbb{R}^d$ satisfying \Cref{as:boundedness} and let $(\varepsilon_i)_{i \in I}$ 
    be i.i.d.\ Rademacher variables, independent of $(\bX_i)_{i \in I}$. Then, for any positive 
    constants $s, \mu$ such that $ 8 \mu  m B^2 \leq 1$,
    \begin{align}
        \mathbb{E}\left[\sup_{w \in s \mathcal{V}} e^{\mu G(\bs w)} \right] \le \exp\big\{(ms^2 B_{{\bs a}^\top \bX}^2)\mu^2 + (5 m B^2 - s^2/2)\mu \big\}.\label{exp:1}
    \end{align}
\end{lem}
Applying \Cref{lem:first_expectation} with $m=n_{k+1}$, $\mu = \lambda = \frac{1}{8n_{k+1}B^2}$ and $s=s_{k+1}$ gives
\begin{align}
    \mathbb{E}\left[\sup_{\bs w \in s_{k+1} \mathcal{V}} e^{\lambda G(\bs w)} \right] &\le \exp\left\{ - \frac{3s_{k+1}^2 - 40n_{k+1}B^2}{64n_{k+1}B^2} \right\}.
\end{align}
The choice of $t_{k+1}$ ensures that $\frac{3s_{k+1}^2 - 40n_{k+1}B^2}{64n_{k+1}B^2} \geq (1+\varepsilon) \ln \ln_\beta n_{k+1} + \ln(1/\delta)$. It follows that
\begin{align}\label{eq:expectation1}
    \mathbb{E}\left[\sup_{\bs w \in s_{k+1} \mathcal{V}} e^{\lambda G(\bs w)} \right] \leq \frac{\delta}{(k+15)^{1+\varepsilon}}.
\end{align}
Finally, summing over all integer $k \geq 0$ and setting $\varepsilon=0.2$, we get
\begin{align}
    \mathbb{P}(\mathcal{A}) \leq \delta \sum_{k\geq 0} \frac{1}{(k+15)^{1+\varepsilon}} \leq 3\delta.
    \end{align}
    
\subsection{Proof of \texorpdfstring{\Cref{thm:UBforBandits}}{Theorem 3}}

In this section, we provide the proof of the upper bound established for
the proposed algorithm in the problem of the best arm identification in the
multi-armed bandit problem. We start with two technical lemmas, then we 
provide two other lemmas that constitute the core technical part of the 
proof of \Cref{thm:UBforBandits}. Finally, in \Cref{ssec6.3.3}, we
put all the pieces together and present the proof of the theorem. 

\subsubsection{Preliminary lemmas}

We state and prove two elementary lemmas which we will need for the proof of \Cref{thm:UBforBandits}.

\begin{lem}\label{lem:first_bandit_inequality}
For $t\geq 1, c>0$ and $0 < \omega \leq 0.15$, we have
\begin{align}
  \frac{1}{t}\ln\left(\frac{\ln(2t)}{\omega} \right) \geq c \implies t \leq \frac{1}{c} \ln\left(\frac{2\ln(1/(c    \omega))}{\omega} \right).
\end{align}
\end{lem}
\begin{proof}
Let $f(t) = \frac{1}{t}\ln\big(\frac{\ln(2t)}{\omega} \big)$, defined for any $t\geq1$ and $t_* = \frac{1}{c} \ln\big(\frac{2\ln(1/(c\omega))}{\omega} \big)$.
It suffices to show that $f(t_*) \leq c$. Indeed, since the function $f$ is decreasing, it implies that $f(t) < c$ for any $t > t_*$ which is the contrapositive of the claimed implication.
Using the definition of $f$ and $t_*$ we have,
\begin{align}
    f(t_*) \leq c &\iff \ln\left(\frac{\ln(2t_*)}{\omega} \right) \leq t_* c\\
    &\iff t_* \leq \frac{1}{2(c\omega)^2}\\
    &\iff \ln\left(\frac{2\ln(1/(c\omega))}{\omega} \right) \leq \frac{1}{2c\omega^2}
\end{align}
The last inequality is clearly true since $\ln(x) \leq \frac{x}{2}$ on $(0, \infty)$ and this proves our claim.
\end{proof}
\begin{lem}\label{lem:second_bandit_inequality}
For $t\geq 1$, $s \geq e$, $c \in (0, 1]$, $0 < \omega \leq \delta \leq e^{-e}$, we have,
\begin{align}
    \frac{1}{t}\ln\left(\frac{\ln(2t)}{\omega}\right) \geq \frac{c}{s}\ln\left(\frac{\ln (s)}{\delta}\right) \implies t \leq \frac{s}{c}\frac{\ln(\nicefrac{2}{\omega}) + \ln \ln (\nicefrac{1}{c\omega})}{\ln(\nicefrac{1}{\delta})}.
\end{align}
\end{lem}
\begin{proof} \Cref{lem:first_bandit_inequality} immediately implies that
\begin{align}
    \frac{c t}{s}  &\leq \frac{\ln(\nicefrac{2}{\omega}) + \ln\left[\ln(s) + \ln(\nicefrac{1}{c\omega}) - \ln\ln(\nicefrac{\ln(s)}{\delta})\right]}{\ln(\nicefrac{1}{\delta}) + \ln \ln (s)}.
\end{align}
Using the fact that $\ln\ln(\nicefrac{\ln(s)}{\delta}) \geq 1$ and the following fact
\begin{align}
    s \geq e & \implies \ln s -1 \geq 0 \\
    & \implies \ln s -1 \le  e(\ln s-1)\\
    &\implies \ln s -1 \le (\ln s-1)\ln(\nicefrac{1}{c\omega}) \\
    &\implies \ln s + \ln(\nicefrac{1}{c\omega})
    -1 \le \ln s\ln(\nicefrac{1}{c\omega})\\
    &\implies \ln s + \ln(\nicefrac{1}{c\omega})
    -\ln\ln(\nicefrac{\ln(s)}{\delta}) \le \ln s\ln(\nicefrac{1}{c\omega}),
\end{align}
we have
\begin{align}
    \frac{ct}{s} \leq \frac{\ln(\nicefrac{2}{\omega}) + \ln\ln(\nicefrac{1}{c\omega}) + \ln\ln s}{\ln(\nicefrac{1}{\delta}) + \ln \ln s}
\end{align}
We conclude by applying the inequality $a\geq b, x > 0 \implies \frac{x+a}{x+b} \leq \nicefrac{a}{b}$ with $a = \ln(\nicefrac{2}{\omega}) + \ln\ln(\nicefrac{1}{c\omega})$, $b = \ln(\nicefrac{1}{\delta})$ and $x = \ln \ln s$.
\end{proof}
\subsubsection{Main lemmas}
Without loss of generality, we assume hereafter that the arms' parameters are ranked in decreasing order : $\theta_1 \geq \theta_2 \geq \dots, \theta_K$.
We define the function\begin{align}
    U(n, \omega) = \frac{3.4\sigma}{\alpha}\sqrt{\frac1n \ln \left(\frac{\ln (2n)}{\omega}\right)}, \quad n \in \mathbb{N}^*,\ \omega \in (0, 1),
\end{align}
and the events
\begin{align}
    \mathcal{E}_k(\omega) = \{ \forall n \geq n_0(\omega) \text{ it holds that } \lvert \hat{\theta}_{k, n} - \theta_k \rvert \leq U(n, \omega) \}.
\end{align}
Note that, according to \Cref{MainThm1d},  $\mathbb{P}\big(\mathcal{E}_k(\omega)^\complement \big) = {O}(\omega)$.
The proof of \Cref{thm:UBforBandits} is essentially the combination of 
two lemmas. The first lemma states that with high probability the number 
of times each sub-optimal arm is pulled is not too large. The second 
lemma shows that the algorithm indeed stops at some time and returns 
the best arm with high probability.
\begin{lem}\label{lem:bandit_1}
Let $\beta\in (0,\frac{2}{\sqrt{2}-1})$, $\delta \in (0, e^{-e})$ and $\varkappa = (2+\beta)^2(\nicefrac{3.4 \sigma}{\alpha})^2$. Then we have, with probability at least $1 - 11 \delta$ and any integer $n \geq 1$,
\begin{align}
    \sum_{k=2}^K T_k(n) \leq n_0(\delta) (K-1) + 104\varkappa \mathbf{H}_1 \ln(\nicefrac{1}{\delta}) + \sum_{k=2}^K \varkappa \frac{\ln(2\max\{1, \ln(\varkappa/(\Delta_k^2\delta))\})}{\Delta_k^2}
\end{align}
\end{lem}

\begin{proof} The proof is carried out in two steps. In the first step, we upper bound the number of pulls on events for which the rewards are well behaved. In the second step we resort on standard concentration arguments to show that the events considered in the first step happen with high probability.

\textbf{Step 1.}
Let $k > 1$. Assuming that $\mathcal{E}_1(\delta)$ and $\mathcal{E}_k(\omega)$ hold true and $I_n = k$, one has, for $n \geq K n_0(\delta)$ (i.e. after warm-up stage),
\begin{align}
    \theta_k + U(T_k(n), \omega) + (1+\beta)U(T_k(n), \delta) &\geq \hat{\theta}_{k, T_k(n)} + (1+\beta) U(T_k(n), \delta) &\text{($\mathcal{E}_k(\omega)$ holds)}\\
    &\geq \hat{\theta}_{1, T_1(n)} + (1+\beta)U(T_1(n), \delta) &\text{($I_n = k$)}\\
    &\geq \theta_1. &\text{($\mathcal{E}_1(\delta)$ holds)}
\end{align} 
Since the function $U$ is decreasing in its second argument, we have
\begin{align}
    (2+\beta)U(T_k(n), \min(\omega, \delta)) \geq \Delta_k \coloneqq \theta_1 - \theta_k.
\end{align}
Setting $\varkappa = (2+\beta)^2(\nicefrac{3.4 \sigma}{\alpha})^2$ and using Lemma~\ref{lem:first_bandit_inequality} with $c = \frac{\Delta_k^2}{\varkappa}$, one obtains that, for $n \geq K n_0(\delta)$, if $\mathcal{E}_1(\delta)$ and $\mathcal{E}_i(\omega)$ hold true and $I_n = k$ then
\begin{align}
    T_k(n) &\leq \frac{\varkappa}{\Delta_k^2}\ln\left(\frac{2\ln(\nicefrac{\varkappa}{\left(\Delta_k^2 \min(\omega, \delta)\right)})}{\min(\omega, \delta)} \right)\\
    &\leq \tau_k + \frac{\varkappa}{\Delta_k^2} \ln\left(\frac{\ln(\nicefrac{e}{\omega})}{\omega} \right)\\
    &\leq \tau_k + \frac{2\varkappa}{\Delta_k^2}\ln\left({1}/{\omega}\right).
\end{align}
with $\tau_k = \frac{\varkappa}{\Delta_k^2} \ln\left( ({2/\delta})\max\{1, \ln(\nicefrac{\varkappa}{\Delta_k^2\delta}) \} \right)$.
Since $T_k(n)$ increases only when $k$ is pulled, the above argument shows that the following inequality is true for any time $n \geq 1$ :
\begin{align}\label{eq:Tkn}
    T_k(n)\mathds{1}\{\mathcal{E}_1(\delta) \cap \mathcal{E}_k(\omega)\} \leq n_0(\delta) + \tau_k + \frac{2\varkappa}{\Delta_k^2}\ln\left({1}/{\omega}\right).
\end{align}
\begin{remark}
Indeed, if arm $k$ is pulled at time $n \geq Kn_0(\delta)$ then 
\begin{align}
    T_k(n+1) - 1 = T_k(n) \leq \tau_k + \frac{2\varkappa}{\Delta_k^2}\ln({1}/{\omega}),
\end{align}
and if arm $k$ is pulled before time $Kn_0(\delta)$, i.e. during the warm-up stage, then
\begin{align}
    T_k(n) \leq n_0(\delta) \leq n_0(\delta) + \tau_k + \frac{2\varkappa}{\Delta_k^2}\ln\left({1}/{\omega}\right).
\end{align}
\end{remark}

\textbf{Step 2.} We define the random variable $\Omega_k \coloneqq \max\{\omega \in [0, 1] : \mathcal{E}_k(\omega) \text{ holds true}\}$. \Cref{MainThm1d} guarantees that it is well defined and that $\mathbb{P}(\Omega_k < \omega) \leq c \omega$ with $c=10.4$\footnote{\Cref{MainThm1d} gives a slightly tighter bound but we chose to loosen it for simplicity of the proof.}. Furthermore, one can rewrite \cref{eq:Tkn} as 
\begin{align}
    T_k(n) \mathds{1}\{\mathcal{E}_1(\delta)\} \leq n_0(\delta) + \tau_k + \frac{2\varkappa}{\Delta_k^2}\ln\left({1}/{\Omega_k}\right)
\end{align}
Therefore, for any $x>0$,
\begin{align}
    \mathbb{P}\left( \sum_{k=2}^K T_k(n) > x + \sum_{k=2}^K (\tau_k + n_0(\delta)) \right) &\leq \mathbb{P}\left(\mathcal{E}_1(\delta)^\complement \right) \\ &\quad+ \mathbb{P}\left(\left\{\sum_{k=2}^K T_k(n) > x + \sum_{k=2}^K (\tau_k + n_0(\delta))\right\} \bigcap  \mathcal{E}_1(\delta) \right)\\ &\leq c\delta +\mathbb{P}\left( \sum_{k=2}^K \frac{2\varkappa}{\Delta_k^2}\ln\left({1}/{\Omega_k}\right) > x  \right)
\end{align}
Define the random variables $Z_k = \frac{2\varkappa}{\Delta_k^2} \ln\left({1}/{\Omega_k}\right)$, for $k \in [K]\backslash \{1\}$. Observe that these are independent non-negative random variables and since $\mathbb{P}(\Omega_k < \omega) \leq c \omega$, it holds that $\mathbb{P}(Z_k > x) \leq c\exp(-x/a_k)$ with $a_k = 2\varkappa/\Delta_k^2$. Observing that
\begin{align}
    \mathbb{E}Z_k = \int_{0}^{+\infty} \mathbb{P}\left(Z_k > x\right)dx \leq c\int_{0}^{+\infty} e^{-x/a_k} = c a_k
\end{align}
and applying a basic concentration inequality for the sum of sub-exponential random variables (see \Cref{lem:sub_exponential_bound}), we have,
\begin{align}
    \mathbb{P}\left(\sum_{k=2}^K(Z_k-ca_k)>z\right) &\leq \mathbb{P}\left(\sum_{k=2}^K(Z_k-\mathbb{E}Z_k)>z\right)\\
    &\leq \exp\left(-\min\left\{\frac{z^2}{8 c \lVert a \rVert_2^2}, \frac{z}{4\lVert a \rVert_\infty}\right\} \right)\\
    &\leq \exp\left(-\min\left\{\frac{z^2}{8c\lVert a \rVert_1^2}, \frac{z}{4\lVert a \rVert_1}\right\} \right).
\end{align}
Putting everything together with $z = 4 c \lVert a \rVert_1 \ln(1/\delta)$, $x = z + c\lVert a \rVert_1$ one obtains, for $n \geq 1$
\begin{align}
    \mathbb{P}\left(\sum_{k=2}^K T_k(n) > \sum_{k=2}^K\left(\frac{10\varkappa c\ln(1/\delta)}{\Delta_k^2} + \tau_k + n_0(\delta) \right) \right) \leq 11\delta
\end{align}
and the claim of the lemma follows.
\end{proof}

\begin{lem}\label{lem:bandit_2}
Let $\beta \in (0, (\nicefrac{2}{\sqrt{2}-1})), \delta \in (0, 0.01)$ and $c_\beta = \big(\frac{2 + \beta}{\beta}\big)^2$. If
\begin{align}
    \lambda \geq \frac{\varrho}{1- 10.4\delta - {\textstyle\sqrt{\delta^{\nicefrac14} \ln(1/\delta)}}},\quad \text{ with }
    \quad \varrho = c_\beta \frac{\ln\left( 2\ln(\nicefrac{c_\beta}{2\delta})/ \delta\right)}{\ln(\nicefrac{1}{\delta})},
\end{align}
then, for all $k=2, \dots, K$ and $n=1, 2, \dots$ we have $T_k(n) < n_0(\delta) + \lambda \sum_{\ell \neq k} T_\ell(n)$ with probability at least $1 - 6 \sqrt{\delta}$.
\end{lem}
\begin{proof}
Let $k > \ell$. Assuming that $\mathcal{E}_k(\omega)$ and $\mathcal{E}_\ell(\delta)$ hold true and that $I_n = k$,  one has, for $n \geq Kn_0(\delta)$,
\begin{align}
    \theta_k + U(T_k(n), \omega) + (1+\beta)U(T_k(n), \delta) &\geq \hat{\theta}_{k, T_k(n)} + (1+\beta)U(T_k(n), \delta)\\
    &\geq \hat{\theta}_{\ell, T_\ell(n)} + (1+\beta)U(T_\ell(n), \delta)\\
    &\geq \theta_\ell + \beta U(T_\ell(n), \delta)
\end{align}
This implies $(2+\beta)U(T_k(n), \min(\omega, \delta)) \geq \beta U(T_\ell(n), \delta)$. Applying \Cref{lem:second_bandit_inequality} with $c=2c_\beta^{-1}$ one obtains that if $\mathcal{E}_k(\omega)$ and $\mathcal{E}_\ell(\delta)$ hold true and $I_n = k$ then
\begin{align}\label{eq:(7)}
    T_k(n) \leq c_\beta \frac{\ln\left( 2\ln(\nicefrac{c_\beta}{2\min(\omega, \delta)}) /\min(\omega, \delta)\right)}{\ln(\nicefrac{1}{\delta})} T_\ell(n).
\end{align}
Since $T_k(n)$ only increases when $k$ is played, then, for all $n\geq 1$,
\begin{align}
    (T_k(n) - n_0(\delta))\mathds{1}\left(\mathcal{E}_k(\omega) \cap \mathcal{E}_\ell(\delta) \right) \leq c_\beta \frac{\ln\left( 2\ln(\nicefrac{c_\beta}{2\min(\omega, \delta)}) /\min(\omega, \delta)\right)}{\ln(\nicefrac{1}{\delta})} T_\ell(n).
\end{align}
Using \eqref{eq:(7)} with $\omega = \delta^{k-1}$ we see that 
\begin{align}
    \mathds{1}\{\mathcal{E}_k(\delta^{k-1})\} \frac{1}{k-1} \sum_{\ell=1}^{k-1} \mathds{1}\{\mathcal{E}_\ell(\delta)\} > 1-\alpha \implies (1-\alpha)(T_k(n) - n_0(\delta)) \leq \varrho \sum_{\ell \neq k}T_\ell(n).
\end{align}
The above implication leads to the following inequalities
\begin{align}
    &\mathbb{P}\bigg(\exists (k, n) \in \{2, \dots, K\}\times \mathbb{N}^* : (1-\alpha)(T_k(n) - n_0(\delta)) \geq \varrho \sum_{\ell \neq k} T_\ell(n) \bigg) \\ & \qquad \leq  \mathbb{P}\bigg( \exists k \in  \{2, \dots, K\} : \mathds{1}\{\mathcal{E}_k(\delta^{k-1})\} \frac{1}{k-1} \sum_{\ell=1}^{k-1} \mathds{1}\{\mathcal{E}_\ell(\delta)\} \leq 1 - \alpha  \bigg)\\
    & \qquad \leq \sum_{k=2}^K \mathbb{P}\Big(\mathcal{E}_k(\delta^{k-1})^\complement \Big) + \sum_{k=2}^K \mathbb{P}\bigg(\frac{1}{k-1}  \sum_{\ell=1}^{k-1} \mathds{1}\left(\mathcal{E}_\ell(\delta) \right) \leq 1 - c\delta - (\alpha - c\delta) \bigg).
\end{align}
Since $\mathbb{E}\mathds{1}\left(\mathcal{E}_\ell(\delta) \right) \geq 1 - c\delta$ with $c=10.4$, using \emph{separately} a union bound and Hoeffding's inequality, we get
\begin{align}
    \mathbb{P}\bigg(\frac{1}{k-1}  \sum_{\ell=1}^{k-1} \mathds{1}\left(\mathcal{E}_\ell(\delta) \right) \leq 1 - c \delta 
    - (\alpha - c\delta) \bigg) \leq \min\big(c(k-1)\delta, 
    \exp(-2(k-1)(\alpha - c\delta)^2\big).
\end{align}
Define $R=e^{-2\delta^{1/4}\ln(1/\delta)}$ and $j= \lceil \ln \{2\delta^{3/4}(1-R)\}/\ln R \rceil$. One can check that $1-R = 1-e^{2\delta^{1/4}\ln \delta} \ge 0.64\delta^{1/4}
\ln(1/\delta)$, which leads to
\begin{align}
    j -1 &\leq - \frac{\ln \{2 \delta^{3/4}(1-R)\}}
      {2\delta^{1/4}\ln(1/\delta)} \le -\frac{\ln \{1.28\delta
\ln(1/\delta)\}}
      {2\delta^{1/4}\ln(1/\delta)} \le  (1/2)\delta^{-1/4}.
\end{align}
Setting $\alpha = c\delta + \sqrt{\delta^{\nicefrac{1}{4}}\ln(1/\delta)}$, we have
\begin{align}
    & \mathbb{P}\bigg(\exists (k, n) \in \{2, \dots, K\}\times \mathbb{N}^* : \big(1- c\delta - {\textstyle\sqrt{\delta^{\nicefrac14}\ln(1/\delta)}}\big)
    \big(T_k(n) - n_0(\delta)\big) \geq \varrho \sum_{\ell \neq k} T_\ell(n) \bigg) \\
    &\leq \sum_{k=2}^K \left\{ c\delta^{k-1} + \min\big(c(k-1)\delta,e^{-2(k-1)\delta^{1/4}\ln(1/\delta)} \big)\right\}\\
    &\leq c \frac{\delta}{1-\delta} + \frac{c\delta}{2}j^2 + \frac{R^j}{1-R} \leq 10.6 \delta + 5.2 \delta j^2 + 2 \delta^{3/4} \leq 6\sqrt{\delta}.
\end{align}
This completes the proof of the lemma.
\end{proof}

\subsubsection{Putting all lemmas together}\label{ssec6.3.3}

Let $\nu$ be  the confidence level from \Cref{thm:UBforBandits} and
let $\delta$ satisfy the relation $\nu = 11\delta + 6\sqrt{\delta}$. 
Note that this implies $\sqrt{\delta} = (\sqrt{11\nu +9}-3)/11$, which
is the value of $\delta$ given in \Cref{algo}. On the one hand,
\Cref{lem:bandit_1} states that, with probability at least $1-11\delta$, 
the total number of times the suboptimal arms are sampled does not 
exceed $(K-1)n_0(\delta) + \varkappa\left(104\mathbf{H}_1 \ln(\nicefrac{1}{\delta}) + \mathbf{H}_2\right)$ where $\varkappa= ((2+\beta)3.4\sigma/\alpha)^2$. On the other hand, \Cref{lem:bandit_2} states that with probability at least $1-6\sqrt{\delta}$, if the parameter $\lambda$ is large enough, only the optimal arm will meet the stopping criterion and therefore, the number of pulls from the optimal arm is equal to $n_0(\delta) + \lambda \sum_{k \geq 2}T_k(n)$. Combining those two lemmas, we have that with probability at least $1-11\delta-6\sqrt{\delta}$, the optimal arm meets the stopping criterion and the total number of pulls does not exceed $(1+\lambda)K n_0(\delta) + (1+\lambda)\varkappa \left( 104\mathbf{H}_1 \ln(\nicefrac{1}{\delta}) + \mathbf{H}_2\right)$.

\subsection{Proof of Theorem \ref{thm:LBforBandits}}

Since $\tilde\phi$ is symmetric, the means of the two arms $\theta_1$ and $\theta_2$ coincide with the parameters of interest and so, the gap $\Delta$ coincides with the difference in means, i.e., $\Delta=|\theta_1-\theta_2|$. Therefore, finding the best arm amounts to finding the arm with the best mean and the result follows from \cite[Corollary 1]{jamieson2014lil},  which in turn is a consequence of
\cite[Theorem 1]{farrell1964asymptotic} which we recall here for completeness 

\begin{theorem} \citealp[Theorem 1]{farrell1964asymptotic}
Let $X_1, X_2,...$ be i.i.d. Gaussian random variables with unknown mean $\Delta\neq0$ and variance $1$. Consider testing whether $\Delta > 0$ or $\Delta < 0$. Let $Y \in \{-1, 1\}$ be the decision of any such test based on $T$ samples (possibly a random number) and let $\delta \in (0, 1/2)$. If $\sup_{\Delta \neq 0}\mathbb{P}\left(Y\neq sign(\Delta)\right)\leq \delta$, then 
\begin{align}
    \limsup\limits_{\Delta \xrightarrow{} 0} \frac{\mathbb{E}_\delta[T]}{\delta^{-2}\ln \ln \Delta^{-2}} \geq 2 - 4 \delta.
\end{align}
\end{theorem}

\section{Proofs of postponed lemmas}\label{sec:postponed_proofs}

\begin{myproof}[Proof of \Cref{lem:1}]
Let $k$ be a positive integer and let $n \in I_k$. We define the vectors
\begin{align}
    \bs{v}_n^* =\frac{\btheta^* - \bm{\hat{\theta}}_n}{{\bs a}^\top(\btheta^* - \bm{\hat{\theta}}_n)} \in \mathcal{V}
    \quad\text{and}\quad
    \bm{\Bar{\theta}}_n = \btheta^* - t_{k+1} \bs{v}^*_n.
\end{align}
Since the sequence $(t(n))_n$ is non-increasing, if $\mathcal{A}_n$ is realized then  $p_n=\frac{t_{k+1}}{{\bs a}^\top(\btheta^* - \bm{\hat{\theta}}_n)} \in (0, 1)$. Furthermore, since $\hat{\Phi}_n$ is a convex function (\Cref{as:convex_lipschitz_phi,as:convex_penalty}) we have,
\begin{align}
\label{eq:bar_theta}
    \inf_{w\in t_{k+1}\mathcal{V}} \hat{\Phi}_n(\btheta^*-\bs{w}) &\le \hat{\Phi}_n(\bm{\Bar{\theta}}_n)=\hat{\Phi}_n(p_n\bm{\hat{\theta}}_n+(1-p_n)\btheta^*)\\
    &\leq p_n \hat{\Phi}_n(\bm{\hat{\theta}}_n) + \left(1 - p_n\right) \hat{\Phi}_n(\btheta^*)\le  \hat{\Phi}_n(\btheta^*).
\end{align}
Therefore, on the event $\mathcal{A}_n$,
\begin{align}
     \sup_{w\in t_{k+1} \mathcal{V}} \left[\hat{\Phi}_n(\btheta^*) - \hat{\Phi}_n(\btheta^*-\bs{w})\right] \geq 0.
\end{align}
We conclude the proof by noting that the curvature of the population risk (\Cref{as:Phi_strongly_convex_everywhere}) implies that for any vector ${\bs w} \in \mathbb{R}^d$,
\begin{align}
    \mathbb{E}\left[\hat{\Phi}_n(\btheta^*) - \hat{\Phi}_n(\btheta^*-\bs{w})\right] &= \Phi_n(\btheta^*) - \Phi_n(\btheta^* - \bs w) \le  - \frac{\alpha_n \lVert {\bs w} \rVert^2}{2}.
\end{align}
\end{myproof}

\begin{myproof}[Proof of \Cref{lem:symmetrization_contraction}] A modified version\footnote{The version we use here can be found, for instance, in \citep[Eq.\ (2.3)]{lecue2014}.} of the symmetrization inequality yields
\begin{align}
\mathbb{E}\left[\sup_{w \in t \mathcal{V}} \exp\left\{ \lambda \left( S_m(w) - \alpha m \lVert w \rVert_2^2 \right) \right\} \right] \leq \mathbb{E} \left[\sup_{{\bs w} \in t\mathcal{V}}\exp\left\{2\lambda(S'_{m}({\bs w})- \alpha m\| {\bs w}\|_2^2 )\right\}\right],
\end{align}
where $S'_{m}({\bs w})$ is the symmetrized version of $S_{m}({\bs w})$, defined by
\begin{align}
    S'_{m}({\bs w}) =\sum_{i=1}^{m} \varepsilon_i \left\{ \phi(Y_i,\bX_i^\top \btheta^*) - \phi(Y_i,\bX_i^\top (\btheta^*-\bs{w}))
    \right\}.
\end{align}
We define the set
$R = \left\{ t \mathbf{X}^\top\bs{v}:
{\bs v} \in \mathcal{V} \right\}\subset \mathbb{R}^m$
and the functions $\varphi_i:\mathbb{R}\to\mathbb{R}$ by
\begin{align}
    \varphi_i : r \mapsto \left[ \phi(Y_i, \bX_i^\top \btheta^*) - \phi(Y_i, \bX_i^\top \btheta^* - r) \right]/L, \quad i=1, \dots, m.
\end{align}
These functions $\varphi_i$ are contractions (\Cref{as:convex_lipschitz_phi}) such that $\varphi_i(0)=0$. The contraction principle \cite[Theorem 2.2]{koltchinskii} gives
\begin{align}
    \mathbb{E} \left[\sup_{{\bs w} \in t\mathcal{V}}\exp\left\{2\lambda(S'_{m}({\bs w})- \alpha m\| {\bs w}\|_2^2 )\right\}\right]\le
    \mathbb{E} \left[\sup_{{\bs w} \in t\mathcal{V}}\exp\left\{2\lambda(L
     \bs{w}^\top \mathbf{X}\bs{\varepsilon}- \alpha m\| {\bs w}\|_2^2 )\right\}\right].
\end{align}
Setting $t'= (\nicefrac{2m\alpha}{L})t$ and $\lambda'=(\nicefrac{ L^2}{m\alpha})\lambda$, we arrive at
\begin{align}\label{exp:5}
    \mathbb{E} \left[\sup_{{\bs w} \in t\mathcal{V}}\exp\left\{2\lambda(S'_{m}({\bs w})- 
    \alpha m\| {\bs w}\|_2^2 )\right\}\right]\le
    \mathbb{E} \left[\sup_{w \in t'\mathcal{V}}\exp\left\{\lambda'(
     \bs{w}^\top \mathbf{X}\bs{\varepsilon}- \lVert {\bs w}\rVert_2^2/2 )\right\}\right].
\end{align}
Finally, since the positive real numbers $\lambda$ and $\lambda'$ are positively proportional, taking the infimum over all positive $\lambda$ is exactly the same as taking the infimum over all positive $\lambda'$.
\end{myproof}

\begin{lem}\label{lem:Xepsilon_exp_ineq}
Let $\mathbf{X}$ be a deterministic $d \times m$ matrix and $\bs{\varepsilon} = (\varepsilon_1, \dots, \varepsilon_m)$ a $m$-dimensional vector with i.i.d.\ Rademacher entries. As soon as $\lVert \mathbf{X} \rVert_F^2 \leq \nicefrac{1}{8}$, we have
\begin{align}
    \mathbb{E}\left[ \exp\left\{ \lVert \mathbf{X} \bs\varepsilon \rVert_2^2 \right\} \right] \leq \exp\left\{ 10 \lVert \mathbf{X} \rVert_F^2 \right\}.
\end{align}
\end{lem}

\begin{myproof}[Proof of \Cref{lem:Xepsilon_exp_ineq}]
 Using the fact that for any positive random variable $\eta$, its expectation can be written as
 $\mathbb{E}[\eta] = \int_0^\infty \mathbb{P}(\eta >z )\,dz$, we get
\begin{align}
\mathbb{E}\left[e^{\lVert  \mathbf{X} \bs{\varepsilon} \rVert_2^2} \right] &\leq e^{2 \lVert \mathbf{X} \rVert_F^2} \mathbb{E}\left[e^{2\left( \lVert \mathbf{X} \bs{\varepsilon} \rVert_2 - \lVert \mathbf{X}  \rVert_F \right)_{+}^2} \right]  \\
&\leq e^{2 \lVert \bX \rVert_F^2} \bigg(1 + \int_{0}^{+\infty} \mathbb{P}\Big( \lVert \mathbf{X} \bs{\varepsilon} \rVert_2 \geq  \lVert \mathbf{X}  \rVert_F + \sqrt{(\nicefrac{1}{2}) \ln(1+z)}  \Big)dz \bigg)
\end{align}
We apply the result from \cite[Example 6.3]{boucheron2013concentration} on the variables $\varepsilon_1 \bX_1, \dots, \varepsilon_m \bX_m$ which are independent zero-mean random variable : setting $c_i = 2 \lVert \bX_i \rVert_2$, we have $\nu = \lVert \mathbf{X} \rVert_F^2$ and therefore, for any $z >0$,
\begin{align}
    \mathbb{P}\Big( \lVert \mathbf{X} \bs{\varepsilon} \rVert_2 \geq  \lVert \mathbf{X}  \rVert_F + \sqrt{(\nicefrac{1}{2}) \ln(1+z)}  \Big) \leq \exp\left\{-\frac{\ln(1+z)}{4\lVert \mathbf{X}\rVert_F^2}\right\}
    = (1+z)^{-\nicefrac{1}{4\lVert \mathbf{X} \rVert_F^2}}\,.\label{exp:3}
\end{align}
Assuming that $\lVert\textbf{X} \rVert_F^2 < \nicefrac{1}{4}$, we can plug this in inequality
\eqref{exp:3} to get
\begin{align}
    \mathbb{E}\left[e^{\lVert  \mathbf{X} \bs{\varepsilon} \rVert_2^2} \right] &\leq e^{2 \lVert \mathbf{X} \rVert_F^2} \left(1 + \frac{4 \lVert \mathbf{X} \rVert_F^2}{1 - 4 \lVert \mathbf{X} \rVert_F^2}\right)\\
    &\leq \exp\left\{ 2 \lVert \mathbf{X} \rVert_F^2 + \frac{4 \lVert \mathbf{X} \rVert_F^2}{1 - 4 \lVert \mathbf{X} \rVert_F^2} \right\}
\end{align}
The RHS of the inequality can be large when $\lVert \mathbf{X} \rVert_F^2$ is close to $\nicefrac{1}{4}$. Restricting $\lVert \mathbf{X} \rVert_F^2 \leq \nicefrac{1}{8}$ we arrive at the desired inequality
$\mathbb{E}\left[e^{\lVert  \mathbf{X} \bs{\varepsilon} \rVert_2^2} \right] \leq \exp\left\{ 10 \lVert \mathbf{X} \rVert_F^2 \right\}$.
\end{myproof}

\begin{myproof}[Proof of \Cref{lem:first_expectation}]
Let us define $\Pi_{{\bs a}^\perp} = \mathbf{I}_d - \bs{a} {\bs a}^\top$ to be the projection matrix onto the orthogonal complement of the vector $a$ and set
\begin{align}
    \bs{w}_* =  \Pi_{{\bs a}^\perp}
    \mathbf{X}\bs{\varepsilon} + s \bs{a}.
\end{align}
One checks that $\bs{w}_*\in s\mathcal{V}$ is the maximizer of the quadratic function $G({\bs w}) = \bs{w}^\top \mathbf{X}\bs{\varepsilon} - \lVert {\bs w} \rVert^2/2$ over the set $s\mathcal{V}$. In addition,
\begin{align}
    G(\bs{w}_*) = \bs{w}_*^\top\mathbf{X}\bs{\varepsilon}-  \| {\bs w}_*\|_2^2/2 =  \frac{1}{2}\left(\big\| \Pi_{{\bs a}^\perp} \mathbf{X}\bs{\varepsilon}\big\|_2^2 + 2s  {\bs a}^\top
    \mathbf{X}\bs{\varepsilon}-  s^2\right).
\end{align}
Denoting by $T(\mu)$ the left hand side of \eqref{exp:1}, we arrive at
\begin{align}
    T(\mu) &\le
     e^{-\mu s^2/2} \mathbb{E}\big[\exp\big\{(\mu \big\|\Pi_{{\bs a}^\perp} \mathbf{X}\bs\varepsilon\big\|_2^2/2 + \mu s {\bs a}^\top\mathbf{X}\bs\varepsilon \big\}\big].
\end{align}
The fact that $\Pi_{{\bs a}^\perp}$ is a contraction and the
Cauchy-Schwarz inequality imply
\begin{align}
    T(\mu) &\le e^{-\mu s^2/2}\big( \mathbb{E}\big[\exp\big\{\mu \|\mathbf{X}\bs\varepsilon\|_2^2 \big\}\big]
    \mathbb{E}\big[\exp\big\{2\mu s  {\bs a}^\top\mathbf{X}\bs\varepsilon \big\}\big]\big)^{1/2}.
    \label{exp:2}
\end{align}
We bound separately the two last expectations.
For the first one,
since ${\mu}\lVert \mathbf{X}\rVert_F^2 \leq {\mu}
m B^2 \leq \nicefrac{1}{8}$,  we can apply~\Cref{lem:Xepsilon_exp_ineq}, conditionally to $\mathbf{X}$ and then integrate w.r.t.\ $\mathbf{X}$, to get
\begin{align}
     \mathbb{E}\big[\exp\big\{{\mu}\lVert \mathbf{X}\bs\varepsilon \rVert_2^2 \big\} \big]
     \leq \mathbb{E} \big[ \exp\big\{ 10\mu \lVert \mathbf{X} \rVert_F^2 \big\} \big]
     \leq \exp\big\{ 10 m \mu  B^2  \big\}.
\end{align}
We now bound the second expectation in the right-hand side of \eqref{exp:2}.
Using the fact that $\varepsilon_{1:m}$ are i.i.d. Rademacher random variables independent from $\mathbf{X}$, as well as the inequality $\cosh(x)\le e^{x^2/2}$, we arrive at
\begin{align}
    \mathbb{E}\big[\exp\big\{(2\mu s)  {\bs a}^\top
    \mathbf{X}\bs\varepsilon \big\}\big]
    \le \mathbb{E}\big[\exp\big\{2(\mu s)^2   \|\mathbf{X}^\top{\bs a}\|_2^2\big\}\big]
    \le \exp\big\{ 2(\mu s)^2 m B_{{\bs a}^\top \bX}^2 \big\}.
\end{align}
Grouping the bounds on these two expectations we obtain the stated inequality.
\end{myproof}

\paragraph{Bounding the sum of random variables with sub-exponential right tails}\label{par:sub_exponential}
\begin{lem}\label{lem:sub_exponential_bound}
Let $X_1, \dots, X_n$ be independent, non-negative, random variables such that there exists positives constants $c$ and $a_1, \dots, a_n$ such that
\begin{align}
    \mathbb{P}\left(X_i > x\right) \leq c e^{-\nicefrac{x}{a_i}}, \qquad x > 0, i=1, \dots, n.
\end{align}
Then, for any real positive $t$,
\begin{align}
    \mathbb{P}\left(\sum_{i=1}^n (X_i - \mathbb{E}X_i) > t\right) &\leq \exp\left(- \min\left(\frac{t^2}{8\lVert a\rVert_2^2}, \frac{t}{4\lVert a \rVert_{\infty}}\right) \right).\\
\end{align}
\end{lem}

\begin{myproof}
Defining $\psi_i(\lambda) \coloneqq \log \mathbb{E}e^{\lambda\left(X_i - \mathbb{E}X_i\right)}, i=1,\dots,n$, Markov inequality and the independence hypothesis give
\begin{align}\label{eq:chernoff}
    \mathbb{P}\left(\sum_{i=1}^n (X_i - \mathbb{E}X_i) > t\right) \leq \inf_{\lambda > 0} e^{-\lambda t}\prod_{i=1}^n e^{\psi_i(\lambda)}.
\end{align}
Using the inequality $\ln u \leq u - 1$ valid for any positive real $u$, we have
\begin{align}
    \psi_i(\lambda) \coloneqq \ln \mathbb{E}e^{\lambda X_i} - \lambda \mathbb{E}X_i \leq \mathbb{E}\left[e^{\lambda X_i} - \lambda X_i - 1 \right].
\end{align}
Let $\phi(u) = e^u - u -1$. The monotone convergence theorem guarantees that for any $\lambda > 0$,
\begin{align}
    \mathbb{E}\phi(\lambda X_i) = \sum_{p \geq 2} \frac{\lambda^p}{p!}\mathbb{E}X_i^p.
\end{align}
Since the $X_i$'s are non-negative, we have, for any integer $p\geq2$ and for any index $i=1,\dots,n$,
\begin{align}
    \mathbb{E}X_i^p &= \int_{0}^{+\infty} \mathbb{P}\left( X_i > t^{1/p}\right) dt \leq c p \int_0^{+\infty} t^{p-1} e^{-\nicefrac{t}{a_i}}dt = c a_i^p p!.
\end{align}
Therefore, for any $\lambda \in (0, \nicefrac{1}{2a_i})$
\begin{align}\label{eq:log_MFG}
    \psi_i(\lambda) \leq \mathbb{E}\phi(\lambda X_i) \leq 2c(\lambda a_i)^2
\end{align}
Plugging \eqref{eq:log_MFG} into \eqref{eq:chernoff} yields
\begin{align}
    \mathbb{P}\left(\sum_{i=1}^n (X_i - \mathbb{E}X_i) > t\right) \leq \inf_{\lambda \in (0, \nicefrac{1}{2a_i})} \exp\left( 2c\lVert a\rVert_2^2\lambda^2 - \lambda t \right).
\end{align}
The minimum is attained in $\lambda^* = \min\left( \frac{t}{4c \lVert a \rVert_2^2}, \frac{1}{2 \lVert a\rVert_{\infty}} \right)$ and yields the stated upper bound
\begin{align}
    \mathbb{P}\left(\sum_{i=1}^n (X_i - \mathbb{E}X_i) > t\right) &\leq \exp\left(- \min\left(\frac{t^2}{8\lVert a\rVert_2^2}, \frac{t}{4\lVert a \rVert_{\infty}}\right) \right).\\
\end{align}

\end{myproof}

\end{document}